\documentclass[a4paper,12pt]{amsart}
\usepackage{amsmath,amsfonts,amssymb,amsthm,amscd}
\usepackage{latexsym,graphicx}
\usepackage[matrix,arrow,ps,color,line,curve,frame]{xy}
\usepackage[usenames,dvipsnames]{color}

\renewcommand{\[}{\begin{equation}}
\renewcommand{\]}{\end{equation}}

\newtheorem{proposition}{\sc Proposition}[section]
\newtheorem{lemma}[proposition]{\sc Lemma}
\newtheorem{corollary}[proposition]{\sc Corollary}
\newtheorem{theorem}[proposition]{\sc Theorem}
\theoremstyle{definition}

\theoremstyle{remark}

\newcommand{\can}{{\sf can}}

\newcommand{\id}{\operatorname{id}}

\renewcommand{\ker}{\operatorname{Ker}}

\def\sw#1{{\sb{(#1)}}}

\newcommand{\ot}{\otimes}
\renewcommand{\phi}{\varphi}
\renewcommand{\epsilon}{\varepsilon}
\renewcommand{\subset}{\subseteq}

\def\t{\sp{[2]}}

\newcommand{\counit}{\varepsilon}
\newcommand{\co}{\,\mathrm{co}\,}
\def\C{{\Bbb C}}
\def\N{{\Bbb N}}
\def\R{{\Bbb R}}
\def\Z{{\Bbb Z}}

\def\id{{\rm id}}

\newcommand{\ls}[1]{\ell(#1)^{\langle 1\rangle}}
\newcommand{\rs}[1]{\ell(#1)^{\langle 2\rangle}}

\newcommand{\ffun}{\mathcal{O}}

\newcommand{\pchoose}[3]{
\left[\!\!
\begin{array}{c}
#1\\
#2
\end{array}
\!\!\right]_{#3}
}
\newcommand{\Span}{\text{Span}}

\newcommand{\pQ}{{\tilde{Q}}}
\newcommand{\qQ}{Q}

\newcommand{\dgz}{{\text{deg}}}

\newcommand{\mxda}{{\text{max deg}_{\mathcal{X}}}}
\newcommand{\mxdb}{{\text{max deg}_{\mathcal{Y}}}}
\newcommand{\mnda}{{\text{min deg}_{\mathcal{X}}}}
\newcommand{\mndb}{{\text{min deg}_{\mathcal{Y}}}}

\usepackage[T1]{fontenc}
\usepackage{stackrel}
\usepackage{stmaryrd}

\newcommand{\T}{{\mathcal{T}}}
\newcommand{\Zn}{{\Z/N\Z }}
\newcommand{\Rt}[1]{\stackbin[#1]{}{\rtimes}}
\renewcommand{\Im}{{\text{Im}}}

\begin{document}
\baselineskip=16pt
\author{Piotr~M.~Hajac}
\address{Instytut Matematyczny, Polska Akademia Nauk, ul.~\'Sniadeckich 8, Warszawa, 00-956 Poland\\
Katedra Metod Matematycznych Fizyki, Uniwersytet Warszawski, ul. Ho\.za 74, Warszawa, 00-682 Poland} 
\email{http://www.impan.pl/\~{}pmh, http://www.fuw.edu.pl/\~{}pmh}
\author{Adam Rennie}
\address{Mathematical Sciences Institute\\
Australian National University\\
Canberra, ACT, 0200 Australia} 
\email{adam.rennie@anu.edu.au}
\author{Bartosz~Zieli\'nski}
\address{Instytut Matematyczny, Polska Akademia Nauk, ul.~\'Sniadeckich 8, Warszawa, 00-956 Poland\\
Department of Theoretical Physics and Computer Science, University of \L{}\'od\'z, Pomorska 149/153 90-236 \L{}\'od\'z, Poland}
\email{bzielinski@uni.lodz.pl}
\title[Heegaard quantum lens spaces]{\large The K-theory of Heegaard quantum lens spaces}
\vspace*{-15mm}\maketitle

{\vspace*{-7mm}\it\large\hfil 
Dedicated to Alan Carey on the occasion of his
60th birthday.\hfil}

\begin{abstract}
Representing $\mathbb{Z}/N\mathbb{Z}$
 as roots of unity, we restrict a natural $U(1)$-action on
the  Heegaard quantum sphere to $\mathbb{Z}/N\mathbb{Z}$, and call the quotient spaces Heegaard quantum lens spaces. 
Then we  use this representation of $\mathbb{Z}/N\mathbb{Z}$ to
construct an
associated complex line bundle. This paper proves the stable
non-triviality
of these line bundles over any of the quantum lens spaces we
consider.
We use the pullback structure of the $C^*$-algebra of the lens space 
to compute its $K$-theory via the Mayer-Vietoris sequence, and an  explicit form of the Bass connecting homomorphism to prove the stable non-triviality
of the  bundles. On the algebraic side we prove the universality of the coordinate algebra of such a lens space 
for a particular set of generators and relations. We also prove the non-existence of non-trivial
invertibles in the coordinate algebra of a lens space. Finally, we prolongate
the $\mathbb{Z}/N\mathbb{Z}$-fibres of 
the Heegaard quantum sphere to $U(1)$, and
determine
the algebraic structure of such a $U(1)$-prolongation.
\end{abstract}

\tableofcontents

\section*{Introduction and preliminaries}

\subsection{Introduction}

It is hard to deny that lens spaces are interesting. Indeed,
they have provided a rich source of examples highlighting subtle phenomena in 
topology. They are simple examples of closed 3-manifolds not determined by 
their homology and fundamental group alone. They  also give examples of 
spaces that might be homotopic but not homeomorphic.  Even  today they still 
provide a fertile arena in which to study topological questions, e.g., 
see~\cite{py03}. 

A typical feature of lens spaces is that they possess non-trivial
line bundles giving rise to torsion in $K$-theory.
 This property of lens spaces remains a 
characteristic feature of their 
quantum analogues, and is a focal point of this paper. In brief, we choose
a particular family of quantum lens spaces, define natural complex line bundles
over them, and prove that they generate torsion in the $K_0$-group.

More precisely, we study a family of three-dimensional lens spaces arising 
from a particular family of quantum 3-spheres, namely the
Heegaard quantum spheres $S^3_{pq\theta}$~\cite{bhms05}. The $C^*$-algebras 
of these Heegaard quantum lens spaces were
defined in \cite{hms06b} as fixed-point subalgebras for a $\Z/N\Z$-action 
obtained by restricting the
natural (diagonal) $U(1)$-action $\alpha$ on the $C^*$-algebra
 $C(S^3_{pq\theta})$ of the Heegaard quantum sphere. Likewise,
we consider fixed-point subalgebras of the coordinate algebra 
$\ffun(S^3_{pq\theta})$ 
of the Heegaard quantum sphere.
We denote the thus obtained coordinate algebras and $C^*$-algebras of these 
quantum lens spaces by $\ffun(L^N_{pq\theta})$ and $C(L^N_{pq\theta})$ 
respectively.

Note that in \cite{z-b05} a different $U(1)$-action was used to define 
another type of Heegaard quantum lens spaces. Both of these types of 
Heegaard quantum lens spaces are  different from those 
quantum lens spaces studied in~\cite{hs03}. The latter are graph $C^*$-algebras
and the former are pullback $C^*$-algebras. This is a crucial technical
difference between these two families of $C^*$-algebras
resulting in application of different tools to study their $K$-theory.

Next, we again represent $\Z/N\Z$ via roots of unity and define the
following associated module 
\[
L_N:=\left\{x\in C(S^3_{pq\theta})\;\left|\;\alpha_{e^{\frac{2\pi i}{N}}}(x)=e^{\frac{2\pi i}{N}}x\right\}\subset C(S^3_{pq\theta})\right.
\]
over $C(L^N_{pq\theta})$. This is a finitely generated projective module
defining a natural complex line bundle for each of our quantum lens spaces.
Our main result can be now summarized as follows.
\begin{theorem} 
The left $C(L^N_{pq\theta})$-module $L_N$ is not stably free, and
$[L_N] - [C(L^N_{pq\theta})]$ generates the torsion part of 
$K_0(C(L^N_{pq\theta}))$.
\end{theorem}

On the way,
we prove that  $\ffun(L^N_{pq\theta})$ is universal for a certain set of 
generators and relations.
Having done this, we show that  $\ffun(S^3_{pq\theta})$  contains no 
invertibles other than non-zero multiples
of the identity. This allows us to prove that the $\ffun(\Z/N\Z)$-comodule 
algebra $\ffun(S^3_{pq\theta})$  is non-cleft, which reflects the
non-triviality of
 the noncommutative $\Z/N\Z$-principal
bundle $S^3_{pq\theta}\rightarrow L^N_{pq\theta}$.
However, to conclude a 
stronger result that
the finitely generated projective $\ffun(L^N_{pq\theta})$-module
\[
\mathcal{L}_N:=\left\{x\in \ffun(S^3_{pq\theta})\;\left|\;\alpha_{e^{\frac{2
\pi i}{N}}}(x)=e^{\frac{2\pi i}{N}}x\right\}\subset \ffun(S^3_{pq\theta})
\right.
\]
is not stably free, we turn to $C^*$-algebras. Representing $L_N$
by an idempotent
with entries in $\ffun(L^N_{pq\theta})$, we infer
the stable non-triviality of $\mathcal{L}_N$ over $\ffun(S^3_{pq\theta})$
from our main result.

The final section proves a quantum version of the classical phenomenon that
$S^3\times_{\Z/N\Z}U(1)\cong S^3\times U(1)$,
in both the algebraic and $C^*$-setting. Namely, we prove that
\begin{align}
C(S^3_{pq\theta})\,\bar{\otimes}\, C(U(1))&\cong (C(S^3_{pq\theta})
\,\bar{\otimes}\, C(U(1)))^{\Z/N\Z},\nonumber\\
\ffun(S^3_{pq\theta})\,{\otimes}\, \ffun{(U(1))}&\cong (\ffun(S^3_{pq\theta})
\otimes \ffun(U(1)))^{\Z/N\Z}.
\end{align}

Here and in what follows, the unadorned tensor product stands
for the algebraic tensor product over the ground field $k$,
 typically of complex
numbers. Since in this paper there is no ambiguity concerning $C^*$-completions
of the algebraic tensor product, we simply use $\bar{\otimes}$ to
denote the completed tensor product. Also, 
we use the convention that, for algebras $A$, $B$, and coalgebras $C$, $D$,
the symbol ${}_A^C\mathrm{Hom}_B^D$ signifies the set of
$k$-linear homomorphisms that are left $A$-linear, right $B$-linear,
left $C$-colinear and right $D$-colinear.

\subsection{Principal comodule algebras}\label{principal} 

The comultiplication, counit
and the antipode of a Hopf algebra $H$  are denoted by $\Delta$, 
$\varepsilon$ and $S$, respectively. 
A right $H$-comodule algebra $P$ is a unital associative algebra  equipped 
with an $H$-coaction 
$ \Delta_P : P \rightarrow P \otimes H$ that is an algebra map.  
For a comodule algebra $P$, we call 
\begin{equation}
P^{\co H}:=\left\{p \in P\,|\,\Delta_P(p)=p \otimes 1\right\}
\end{equation}
the subalgebra of coaction-invariant elements in $P$. A left coaction
on $V$ is denoted by ${}_V\Delta$.
For comultiplications and coactions, 
we often employ the Heynemann-Sweedler
notation with the summation symbol suppressed: 
\begin{equation}
\Delta(h)=:h_{(1)}\otimes h_{(2)},\quad
\Delta_P(p)=:p_{(0)}\otimes p_{(1)},\quad
{}_V\Delta(v)=:v_{(-1)}\otimes v_{(0)}.
\end{equation}
With this notation, the {\em convolution product} of maps
$f$ and $g$ from a coalgebra to an algebra is given by
$(f*g)(h):=f(h_{(1)})g(h_{(2)})$.

If $M$ is a right comodule over a coalgebra $C$ and $N$ is a left $C$-
comodule,
then we define their {\em cotensor product} as
\[
M\underset{C}{\Box}N:=\{t\in M\otimes N\;|\;(\Delta_M\otimes\id)(t)=
(\id\otimes{}_N\Delta)(t)\}.
\]
In particular, for a right $H$-comodule algebra $P$ and a left 
$H$-comodule $V$, we observe that $P\Box_H V$ is a left $P^{\co H}$-
module
in a natural way.
Furthermore, if $V$ is a Hopf algebra with comultiplication 
$\tilde{\Delta}$, a Hopf algebra
surjection $\pi:V\rightarrow H$, and a left coaction 
${}_V\Delta:=(\pi\ot\id)\circ\tilde{\Delta}$, 
then $P\Box_H V$ becomes a $V$-comodule
algebra for the coaction $\id\ot\tilde{\Delta}$.
 

An $H$-comodule algebra $P$ is called  {\em principal} 
\cite{bh04} if: 
\begin{enumerate}
\item
$P{\otimes}_B P\ni p \otimes q \mapsto\can(p \otimes q):=
pq_{(0)} \otimes q_{(1)}\in P \otimes H$
is bijective;
\item
$\exists s\in {}_B\mathrm{Hom}^H(P,B\otimes P):\; m\circ s=\id$,
where $m$ is the multiplication map;
\item
the antipode of $H$ is bijective.
\end{enumerate}
Here (1) is the Hopf-Galois (freeness) condition, (2) means equivariant
projectivity of $P$, and (3) ensures a left-right symmetry of the
 definition (everything can be re-written for left comodule
algebras).
The inverse of the map $\can$ can be written explicitly using 
Heynemann-Sweedler like notation: 
$\can^{-1}(p\otimes h):=ph\sp{[1]}\otimes_B h\t$. Here the map
\begin{equation}
H\ni h\longmapsto\can^{-1}(1\otimes h)
=:h\sp{[1]}\underset{B}{\otimes} h\t
\in P\underset{B}{\otimes}P
\end{equation}
is called a {\em translation map}, and satisfies
$h\sp{[1]} h\t=\counit(h)$.

One of the key properties of principal comodule algebras is that,
for any finite-dimensional left $H$-comodule $V$, the left $P^{\co H}$-module
$P\Box_HV$ is finitely generated projective~\cite{bh04}. Here $P$ plays the
role of a principal bundle and $P\Box_HV$ plays the role of an associated
vector bundle. Therefore, we call $P\Box_HV$ an \emph{associated module}.
On the other hand, if $V$ is a Hopf algebra and 
$P\Box_HV$ is a $V$-comodule algebra as described above,
then (under some minor technical assumptions) the principality of
the $H$-coaction on $P$ implies the principality of the $V$-coaction on
$P\Box_HV$ (see~\cite{qsng}). 
We call the principal comodule algebra $P\Box_HV$
the \emph{$V$-prolongation} of~$P$ because it is a direct analogue
of a prolongation of a principal bundle that is obtained by 
enlarging its structure group.   

If $H$ is a Hopf algebra with bijective antipode and $P$ is 
a right $H$-comodule algebra, then one can show (cf.~\cite{bh04})
that it is
principal if and only if there exists a linear map
\begin{equation}
\ell:H\longrightarrow P\otimes P, \quad h\longmapsto
\ell(h)=:\ls{h}\otimes \rs{h},
\end{equation}
such that, for all $h\in H$, $\ell$ satisfies the three equations
\begin{gather}
\ls{h}\rs{h}\sw{0}\otimes\rs{h}\sw{1}=1\otimes h,\\
S(h\sw{1})\otimes\ls{h\sw{2}}\otimes \rs{h\sw{2}}=
\ls{h}\sw{1}\otimes\ls{h}\sw{0}\otimes\rs{h},\\
\ls{h\sw{1}}\otimes\rs{h\sw{1}}\otimes h\sw{2}
=\ls{h}\otimes\rs{h}\sw{0}\otimes\rs{h}\sw{1}.
\end{gather}
Any such  map $\ell$ can be made unital \cite{bh04}. It is 
then called a {\em strong connection} \cite{h-pm96,dgh01,bh04}, 
and can be thought of as an appropriate lifting of 
the translation map.
Given a strong connection, we can explicitly compute an idempotent representing
the module $P\Box_HV$~\cite{bh04}. In particular, if $\dim V=1$, the coaction
${}_V\Delta$ is determined by a group-like $g\in H$, i.e., 
${}_V\Delta(1)=g\otimes 1$. Then, in order to obtain an idempotent representing
$P\Box_HV$, we write $\ell(g)=\sum_i x_i\otimes e_i$, where elements  $e_i$
are chosen
to be linearly independent. The desired idempotent matrix is given by
$e_{ij}:=e_ix_j$.

A special class of principal comodule algebras is distinguished by the 
existence of a cleaving map. A cleaving map is defined as a unital 
right $H$-colinear
convolution-invertible map \mbox{$j:H\rightarrow P$}. Having a cleaving map,
one can define a strong connection as
$
\ell:=(j^{-1}\otimes j)\circ\Delta
$,
where $j^{-1}$ stands for the convolution inverse of~$j$.
Comodule algebras admitting
a cleaving map are called {\em cleft}. In particular, if $j$ is
a colinear algebra homomorphism, it is a cleaving map (not the
other way round). 
In this special case a cleaving map serves as an analogue of a 
trivialisation of a principal bundle.
 Therefore, we can refer to comodule algebras admitting a cleaving
map that is an algebra homomorphism as trivial comodule algebras. 
Note that proving the non-cleftness 
of a principal comodule algebra is stronger
than proving its non-triviality.

All modules associated with cleft
comodule algebras are always free. Also, one can show that a cleaving map is
automatically injective. Therefore, as the value of a cleaving map on
a group-like element is invertible, we can conclude that the existence of
a non-trivial group-like in $H$ necessitates the existence of an invertible
element in $P$ that is not a multiple of~$1$. Hence one of the ways to prove
the non-cleftness of a principal comodule algebra over a Hopf algebra with a 
non-trivial group-like is to show the lack of non-trivial invertibles in
the comodule algebra.

\subsection{From quantum disc to quantum lens spaces}

\subsubsection{Quantum disc}
\label{subsub:Quantum disc}

A two-parameter family of quantum unit discs was defined in 
\cite{Kl:Disc}. Here we consider the one 
parameter subfamily studied therein. We start with a coordinate $\ast$-algebra $\ffun(D_p)$ generated by 
a single element $x$ and the relation
\begin{equation}
x^\ast x-pxx^\ast=1-p,\ \ 0\leq p<1.\label{disc}
\end{equation} 

We can introduce another algebra $\ffun^-(D_p)$ generated by $x_{-}$ with relation 
\begin{equation}
x_{-}^\ast x_{-}-p^{-1}x_{-}x_{-}^\ast=1-p^{-1}.
\end{equation} 
Then assignment $x\mapsto x_{-}^{\ast}$ can be extended to a $\ast$-algebra isomorphism
\begin{equation}
\kappa_p:\ffun(D_p)\rightarrow\ffun^-(D_{p}).\label{kappa}
\end{equation}


Let us denote for brevity
$X:=(1-xx^*)$, so that $1-x^*x=pX$. It follows from (\ref{disc}) that $Xx=pxX$,
and more generally
\begin{equation}
X^kx^n=p^{kn}x^nX^k,\qquad X^kx^{*n}=p^{-kn}x^{*n}X^k,\quad 
n,k\in\N,
\label{1mxx}
\end{equation} 
where the second equation follows from the self-adjointness of $X$. The 
universal $C^*$-algebra for the relation
\eqref{disc} contains $\ffun(D_p)$ and is isomorphic with
the Toeplitz algebra ${\mathcal T}$  
 for all $0\leq p<1$~\cite{Kl:Disc}. In particular, we can take
the relation \eqref{disc} with $p=0$ as a convenient presentation for 
the $C^*$-algebra ${\mathcal T}$. Then \eqref{disc} reduces to
$x^*x=1$, so that $x$ becomes an isometry.

\subsubsection{Heegaard quantum sphere}

For $0\leq p,q,\theta<1$, $\theta$ irrational, the coordinate
algebra of the Heegaard quantum sphere $\ffun(S^3_{pq\theta})$ 
\cite{bhms05}
is the universal $*$-algebra generated  by two elements $a$ and $b$ 
satisfying the relations
\begin{subequations}
\label{heegard}
\begin{gather}
ab=e^{i2\pi\theta}ba,\quad ab^*=e^{-i2\pi\theta}b^*a,\\
a^*a-paa^*=1-p,\quad b^*b-qbb^*=1-q,\\
(1-aa^*)(1-bb^*)=0.
\end{gather}
\end{subequations}
Recall that $\{a,a^*\}$ and $\{b,b^*\}$ generate algebras 
$\ffun(D_p)$ and 
$\ffun(D_q)$ respectively (see \cite[(2.44)]{bhms05}).
Furthermore, by Subsection~\ref{subsub:Quantum disc}, for
$A:=(1-aa^*)$, $B:=(1-bb^*)$,
 we have the relations
\begin{equation}
Aa=paA,\quad Ab=bA,\quad Ba=aB,\quad Bb=qbB,\quad A^*=A,\quad B^*=B.
\end{equation}
Now we can write a basis  of 
$\ffun(S^3_{pq\theta})$ \cite{bhms05} as 
\begin{equation}\label{hbasis}
\{A^ka^\mu b^\nu\;|\;k\geq 0,\mu,\nu\in\Z\}\cup 
\{B^ka^\mu b^\nu\;|\;k> 0,\mu,\nu\in\Z\}.
\end{equation}
Here for $\mu,\nu<0$ we have written $b^\mu=b^{*|\mu|}$ and 
$a^\nu:=a^{*|\nu|}$ for brevity. The $C^*$-algebra of the Heegaard
quantum sphere $C(S^3_{pq\theta})$ can also be defined as the universal
$C^*$-algebra for the relations~\eqref{heegard}. One can set the 
parameters $p$ and $q$ equal to zero without changing this 
$C^*$-algebra, and prove that it is isomorphic with a certain pullback
$C^*$-algebra~\cite{bhms05}.

Let $\ffun(U(1))$ be the coordinate $*$-Hopf algebra of $U(1)$ 
generated by a unitary $u$. 
The coaction of $\ffun(U(1))$ on $\ffun(S^3_{pq\theta})$ is defined on 
generators by
$\rho(a)=a\otimes u$, $\rho(b)=b\otimes u$.
This coaction defines a $\Z$-grading $\text{deg}:\ffun(S^3_{pq\theta})\rightarrow\Z$ 
on $\ffun(S^3_{pq\theta})$, with $\text{deg}(a)=1=\text{deg}(b)$. 
Note that all the basis elements in \eqref{hbasis} have a definite grading degree. The coaction $\rho$ can be equivalently written
as an action 
\[
\alpha: U(1)\longrightarrow \mathrm{Aut}(\ffun(S^3_{pq\theta})),
\quad\alpha_{e^{i\varphi}}(a)=e^{i\varphi}a,
\quad\alpha_{e^{i\varphi}}(b)=e^{i\varphi}b.
\]
This action extends to the $C^*$-algebra 
$C(S^3_{pq\theta})$. One can 
prove that the algebra of coaction-invariant (or action-invariant)
elements is generated as a $*$-algebra
by  $A$, $B$, and $z=ab^*$. They satisfy the relations
\begin{subequations}
\label{lenseB}
\begin{gather}
A^*=A,\quad B^*=B,\quad AB=0,\quad
Az=pzA,\quad zB=qBz,\label{lenseba}\nonumber\\
z^*z=1-pA-B,\quad zz^*=1-A-qB\label{lensebb}.
\end{gather} 
\end{subequations}
The universal $*$-algebra for these relations
coincides with the coaction-invariant
subalgebra. We call it
 the coordinate algebra of a mirror quantum sphere~\cite{hms06b}.
Note that by \cite{HKMZ},  $\ffun(S^3_{pq\theta})$ is a piecewise trivial principal comodule algebra.
The covering is given by a pair of ideals $\ffun(S^3_{pq\theta})A$ and $\ffun(S^3_{pq\theta})B$. The quotients 
$\ffun(S^3_{pq\theta})/\ffun(S^3_{pq\theta})A$ and $\ffun(S^3_{pq\theta})/\ffun(S^3_{pq\theta})B$ are both given by
quantum solid tori~\cite{bhms05}.

\subsubsection{Heegaard quantum lens spaces \cite{hms06b,z-b05}}
\label{subsub:lens}

The $*$-Hopf algebra  $\ffun(\Z/N \Z)$ is generated by a unitary
 element $\tilde{u}$ satisfying $\tilde{u}^N=1$. 
There is a natural surjection $\pi:\ffun(U(1))\rightarrow \ffun(\Z/N\Z)$
 given by $u\mapsto \tilde{u}$. This surjection defines a coaction
of $\ffun(\Z/N\Z)$ on $\ffun(S^3_{pq\theta})$. The coaction-invariant 
subspace for this $\ffun(\Z/N\Z)$-coaction is simply the subspace of 
elements of
degree divisible by~$N$. We denote this algebra by 
$\ffun(L^N_{pq\theta})$, and  call $L^N_{pq\theta}$ the Heegaard 
quantum lens space of type~$N$. 
Likewise, we inject $\Z/N\Z$ into $U(1)$ via roots of unity, and use
the action $\alpha$ to define the $\Z/N\Z$-invariant subalgebra
of the $C^*$-algebra $C(S^3_{pq\theta})$. We call the invariant
subalgebra the $C^*$-algebra
of the Heegaard 
quantum lens space of type~$N$, and denote by~$C(L^N_{pq\theta})$.

\subsection{The Bass connecting homomorphism}

Consider a pullback diagram  
\begin{equation}
\mbox{$\xymatrix@=5mm{& & A \ar[lld]_{\mathrm{pr}_1}
\ar[rrd]^{\mathrm{pr}_2} & &\\
A_1 \ar@{>>}[drr]_{\pi_1}& & & &A_2 \ar[dll]^{\pi_2}\\
&& A_{12} &&}$}
\end{equation}
in the category of unital
algebras. Explicitly, we can write 
\[
A\cong\{(a_1,a_2)\in A_1\times A_2\;|\;\pi_1(a_1)=\pi_2(a_2)\}
=\ker\left(
A_1\oplus A_2\stackrel{(\pi_1,-\pi_2)}{\longrightarrow} A_{12}
\right).
\] 
If one of the defining
morphisms (here we choose~$\pi_1$) is surjective,
then there exists a
long exact sequence in algebraic $K$-theory \cite{b-h68}
\begin{equation}\label{long}
\cdots\longrightarrow {K}_1^{\text{alg}}(A_1\oplus A_2)\longrightarrow
{K^{\text{alg}}}_1(A_{12})
\stackrel{\mathrm{\tiny Bass}^{\text{alg}}}{\longrightarrow}
{K}_0^{\text{alg}}(A)\longrightarrow {K}_0^{\text{alg}}(A_1\oplus A_2)
\longrightarrow {K}_0^{\text{alg}} (A_{12}).
\end{equation}
The mapping 
$\mathrm{Bass}^{\text{alg}}:{K}_1^{\text{alg}}(A_{12})\longrightarrow
{K}_0^{\text{alg}}(A)$ is obtained as follows. 
Take  an 
invertible matrix $U\in {GL}_n(A_{12})$ representing a class
in $K_1^{\text{alg}}(A_{12})$. 
There exist liftings ${c},{d}\in{{M}}_n(A_{1})$ such
that $\pi_1({c})=U^{-1}$ and $\pi_1({d})=U$.
Then (e.g., see \cite{DHHM}) 
\begin{equation}\label{p}
 p_U:=
  \left(
  \begin{array}{cc}
  ({c}(2-{d}{c}){d},1 ) & ({c}(2-{d}{c})(1-{d}{c}),0 ) \cr
  ((1-{d}{c}){d},0 ) & ((1-{d}{c})^2,0  )
  \end{array}
  \right)\in {{M}}_{2n}(A).
 \end{equation}
is an idempotent matrix. The assignment 
\[\label{bassproj}
\mathrm{Bass}^{\text{alg}}: K_1^{\text{alg}}(A_{12})\ni [U]\longmapsto 
[p_U]-[I_n] \in K_0^{\text{alg}}(A),
\]
where $I_n$ is the
identity matrix of the same size as the matrix~$U$, gives
 the Bass connecting homomorphism~\cite[Theorem~3.3, Page~28]{m-j71}.

It is known that the Bass connecting
homomorphism exists also for the $K$-theory of $C^*$-algebras
(cf.~\cite{gh04}), and
is given by the same explicit formula. 
Since this formula is pivotal in proving our main result,
for the sake of completeness, we provide its complete proof
assuming it for the algebraic $K$-theory\footnote{
This proof is a courtesy of Nigel Higson.}. We proceed by translating
the final part of the long exact sequence \eqref{long} into the 
$K$-theory of $C^*$-algebras. The $K_0$-groups are simply the same,
and comparing the definitions of $K_1^{\text{alg}}$ and $K_1$ immediately
yields a functorial surjection \mbox{$K_1^{\text{alg}}(A)\!\ni\![U]
\mapsto [U]\!\in\! K_1(A)$}
for any unital $C^*$-algebra~$A$.

Next, we want to split this surjection and define 
\mbox{$K_1(A_{12})\!\stackrel{\mathrm{Bass}}{\longrightarrow}\!K_0(A)$}
 by
composing such a set-theoretical 
splitting with $\mathrm{Bass}^{\text{alg}}$. In order
to show that it is independent of the choice of a splitting, we need to
use the homotopy invariance of~$K_0$. More precisely, if 
$[U_0]=[U_1]\in K_1(A_{12})$, then there exists $n\in\N$ such that
$\widetilde{U_0}:=\mathrm{diag}(U_0,I_k)$ and 
$\widetilde{U_1}:=\mathrm{diag}(U_1,I_l)$ are 
elements of $GL_n(A_{12})$ that are homotopic via 
 elements of~$GL_n(A_{12})$. In other words, there exists an invertible
 element $U$ in the
$C^*$-algebra $C([0,1],M_n(A_{12}))
\cong M_n(C([0,1],A_{12}))\cong
M_n(A_{12}\,\bar{\otimes}\,C([0,1]))$ satisfying 
$\mathrm{ev}_0(U)=\widetilde{U_0}$ and
$\mathrm{ev}_1(U)=\widetilde{U_1}$, 
where $\mathrm{ev}_t$ stands for the evaluation
map at~$t$. Furthermore, since tensoring with nuclear $C^*$-algebras is
exact, we can conclude that $A\,\bar{\otimes}\,C([0,1])$ 
is isomorphic  with the pullback
\[
\ker\left((\pi_1,-\pi_2)\otimes\id:
(A_{1}\,\bar{\otimes}\,C([0,1]))\oplus (A_{2}\,\bar{\otimes}\,C([0,1]))
\longrightarrow A_{12}\,\bar{\otimes}\,C([0,1])
\right).
\]
This allows us to
 apply the Bass construction to $U$ to obtain an idempotent
$p_U$ in  the $C^*$-algebra
$M_{2n}(A\bar{\otimes}C([0,1]))$. On the other hand,
the evaluation maps $\mathrm{ev}_0,\mathrm{ev}_1:
A\,\bar{\otimes}\,C([0,1])\rightarrow A$ are homotopic, so that, by
the homotopy invariance of $K_0$, we conclude that 
$[p_{\widetilde{U_0}}]={\mathrm{ev}_0}_*[p_U]={\mathrm{ev}_1}_*[p_U]
=[p_{\widetilde{U_1}}]\in K_0(A)$. Consequently, 
we obtain
\[
\mathrm{Bass}^{\text{alg}}([U_0])=
\mathrm{Bass}^{\text{alg}}([\widetilde{U_0}])=
[p_{\widetilde{U_0}}]-[I_n]=
[p_{\widetilde{U_1}}]-[I_n]=
\mathrm{Bass}^{\text{alg}}([\widetilde{U_1}])=
\mathrm{Bass}^{\text{alg}}([U_1]).
\]

Thus we have 
defined a map \mbox{$\mathrm{Bass}:
K_1(A_{12})\!\stackrel{\mbox{\tiny\rm lift}}
{\longrightarrow}\!
K_1^{\mathrm{alg}}(A_{12})\!\stackrel{\mathrm{Bass}^{\text{alg}}}
{\longrightarrow}\!K_0^{\mathrm{alg}}(A)=K_0(A)$}. Since
$[\mathrm{diag}(U,U')]\in K_1^{\text{alg}}(A_{12})$ 
is a lifting of $[\mathrm{diag}(U,U')]\in K_1(A_{12})$, the
map $\mathrm{Bass}$ is automatically a group homomorphism.
This leads to the the following diagram:
\begin{equation}
\vcenter{
\xymatrix{
K_1^{\text{alg}}(A_1)\oplus K_1^{\text{alg}}(A_2)
\ar[rr]^-{{\pi_{1}}_{\ast\text{alg}}-{\pi_{2}}_{\ast\text{alg}}}
\ar@<-1cm>[d]\ar@<1cm>[d]&&
K_1^{\text{alg}}(A_{12})\ar[r]^-{\mathrm{Bass}^{\text{alg}}}\ar[d]&
K_0(A)
\ar[rr]^-{({\mathrm{pr}_1}_\ast , {\mathrm{pr}_2}_\ast)}\ar@{=}[d]&&
K_0(A_1)\oplus K_0(A_2)\ar@{=}[d]\\
K_1(A_1)\oplus K_1(A_2)
\ar[rr]^-{{\pi_{1}}_{\ast}-{\pi_{2}}_{\ast}}&& 
K_1(A_{12})\ar[r]^-{\mathrm{Bass}}&
K_0(A)\ar[rr]^-{({\mathrm{pr}_1}_\ast , {\mathrm{pr}_2}_\ast)}&&
K_0(A_1)\oplus K_0(A_2).
}
}
\end{equation}
Here the vertical arrows are canonical surjections. The commutativity
of the first two squares follows from the functoriality of these 
surjections. The remaining two squares are commutative by construction.
Now, the exactness of the top row (see~\eqref{long}) and the surjectivity
of all vertical arrows imply the exactness of the bottom row. 
Combining this with the Bott periodicity, we obtain
the Mayer-Vietoris 6-term exact sequence \cite{s-c84,bm} 
\begin{equation}\label{mv}
     	\begin{CD}
	{{K}_0 (A)} @ >{({\mathrm{pr}_1}_\ast , {\mathrm{pr}_2}_\ast)}>> {{K}_0 ( A_1 )
\oplus {K}_0 (A_2)} @ >
        {{\pi_1}_\ast - {\pi_2}_\ast} >> {{K}_0 (A_{12} )} @
	. @  .\\
 	@ A{\mathrm{Bass}}AA @ . @ VV{}V @ .\\
	{{K}_1 (A_{12}) } @ <{{\pi_1}_\ast - {\pi_2}_\ast} 
        << {{K}_1 (A_1 ) \oplus {K}_1 (A_2)} @ <
         {({\mathrm{pr}_1}_\ast , {\mathrm{pr}_2}_\ast)} << {{K}_1 (A)}\ . @ . @ .\\
     	\end{CD}
\end{equation}

\section{Comodule algebras over the coordinate
algebras of Heegaard  lens spaces}
\setcounter{equation}{0}

\subsection{Polynomial identities for the quantum disc}\label{quantumdisc}

Recall the definition of $p$-deformed binomial coefficients
\begin{equation}
\pchoose{n}{m}{p}:=\frac{[n]_p!}{[m]_p![n-m]_p!}, \label{pchoosedef}
\end{equation}
where $p$-deformed factorials
$[n]_p!=[1]_p[2]_p\ldots[n-1]_p[n]_p \text{ for } n-1\in\N,\ 
[0]_p!=1$,
are defined in terms of $p$-deformed naturals
$[n]_p=1+p+p^2+\ldots +p^{n-2}+p^{n-1} \text{ for } n-1\in\N, \ 
[0]_p=0$.

Let $Y$ be a variable. Define a family of polynomials in $Y$,
for $p\in\R_+$ and $n-1\in\N$, by the formulae
\begin{equation}
\label{qdef}
\pQ^p_n(Y)=\sum_{m=1}^n(-1)^mp^{-nm+\frac{m(m+1)}{2}}\pchoose{n}{m}{p}Y^m.
\end{equation}
\begin{lemma}
\label{lem:first-poly}
The polynomials~\eqref{qdef} are uniquely determined by the  recursive equations
\begin{equation}\label{qdefrec}
\pQ^p_1(Y)=-Y,\quad \pQ^p_{n+1}(Y)=(1-Y)\pQ^p_n(p^{-1}Y)-Y.
\end{equation}
\end{lemma}
\begin{proof}
We will proceed by induction.
It follows from the definition of $\pQ^p_n$, Equation~\eqref{qdef}, that 
the case $n=1$ is satisfied. It is useful to rewrite the right-hand 
side of the second equation 
of \eqref{qdefrec} as follows
$
(1-Y)\pQ^p_n(p^{-1}Y)-Y=\pQ^p_n(p^{-1}Y)-Y\pQ^p_n(p^{-1}Y)-Y.
$
For the first term we will separate the $m=1$ term from the sum defining $\pQ^p_n(p^{-1}Y)$, 
while for the second we will separate the $m=n$ term, and then renumber the sum. This yields 
\begin{align}
&(1-Y)\pQ^p_n(p^{-1}Y)-Y\nonumber\\
&=(1-Y)\left(
\sum_{m=1}^n(-1)^mp^{-nm+\frac{m(m+1)}{2}}\pchoose{n}{m}{p}p^{-m}Y^m
\right)-Y\nonumber\\
&=-p^{-n}\pchoose{n}{1}{p}Y-Y-(-1)^np^{-\frac{n(n+1)}{2}}\pchoose{n}{n}{p}Y^{n+1}\nonumber\\
&\quad+\sum_{m=2}^n\left(
(-1)^mp^{-nm+\frac{m(m+1)}{2}}\pchoose{n}{m}{p}p^{-m}-(-1)^{m-1}p^{-n(m-1)+\frac{m(m-1)}{2}}\pchoose{n}{m-1}{p}p^{-(m-1)}\right)Y^m \nonumber\\
&=-p^{-n}\left(\pchoose{n}{1}{p}+p^n\pchoose{n}{0}{p}\right)-(-1)^np^{-\frac{n(n+1)}{2}}\pchoose{n+1}{n+1}{p}Y^{n+1}\nonumber\\
&\quad +\sum_{m=2}^n(-1)^mp^{-(n+1)m+\frac{m(m+1)}{2}}
\left(
\pchoose{n}{m}{p}+p^{n+1-m}\pchoose{n}{m-1}{p}
\right)Y^m.\nonumber
\end{align}
At this point we recall that for all $n\geq 0$, $m>0$, the deformed binomial coefficients satisfy the  recursive formula 
\begin{equation}
\label{pchrec}
\pchoose{n}{n}{p}=\pchoose{n}{0}{p}=1,\quad \pchoose{n+1}{m}{p}=\pchoose{n}{m}{p}+p^{n+1-m}\pchoose{n}{m-1}{p}.
\end{equation}

Applying this to our computation we obtain
\begin{align}
&-p^{-n}\pchoose{n+1}{1}{p}\!Y
+\!\sum_{m=2}^n(-1)^mp^{-(n+1)m+\frac{m(m+1)}{2}}\pchoose{n+1}{m}{p}\!Y^m-(-1)^np^{-\frac{n(n+1)}{2}}\pchoose{n+1}{n+1}{p}\!Y^{n+1}\nonumber\\
&=\sum_{m=1}^{n+1}(-1)^mp^{-(n+1)m+\frac{m(m+1)}{2}}\pchoose{n+1}{m}{p}Y^m=\qQ^p_{n+1}(Y).\nonumber
\end{align}
This completes the proof.
\end{proof}

\begin{lemma}
For all $m,n\in\N\setminus\{0\}$, 
the family of polynomials $\{\pQ^p_k\}_k$ satisfies
\begin{equation}
\label{qmneq}
\pQ^p_{m+n}(Y)=(1+\pQ^p_m(Y))\pQ^p_n(p^{-m}Y)+\pQ^p_m(Y).
\end{equation}
\end{lemma}
\begin{proof}
We prove the above formula  for arbitrary $n\in\N$ by induction on $m$.
The case $m=1$ is true by Lemma~\ref{lem:first-poly}.
For the inductive step, suppose that Equation~\eqref{qmneq} is satisfied for some $m>0$. Then using Lemma~\ref{lem:first-poly}
yields
\begin{align}
&(1+\pQ^p_{m+1}(Y))\pQ^p_{n}(p^{-(m+1)}Y)+\pQ^p_{m+1}(Y)\nonumber\\
&\qquad=(1+(1-Y)\pQ^p_m(p^{-1}Y)-Y)\pQ^p_n(p^{-1-m}Y)+(1-Y)\pQ^p_m(p^{-1}Y)-Y\nonumber\\
&\qquad=(1-Y)\left((1+\pQ^p_m(p^{-1}Y))\pQ^p_n(p^{-1-m}Y)+\pQ^p_m(p^{-1}Y)\right)-Y\nonumber\\
&\qquad=(1-Y)\pQ^p_{n+m}(p^{-1}Y)-Y\nonumber\\
&\qquad=\pQ^p_{n+m+1}(Y),
\end{align}
as desired.
\end{proof}

We now define  polynomials for all $\mu\in\Z$ by the formulae
\begin{equation}
\label{qzdef}
\qQ^p_\mu(Y)=\left\{\begin{array}{lcr}
\pQ^p_\mu(Y) & \text{if} & \mu>0\\
0          & \text{if} & \mu=0\\
\pQ^{p^{-1}}_{-\mu}(pY) &\text{if} & \mu<0
\end{array}\right.\;.
\end{equation}
Note that the polynomials $\qQ^p_{-\mu}$ for $\mu>0$ satisfy the  recursive relations
\begin{equation}
\label{qminrec}
\qQ^p_{-1}(Y)=-pY,\qquad
\qQ^p_{-\mu-1}(Y)=(1-pY)\qQ^p_{-\mu}(pY)-pY.
\end{equation}

\begin{lemma}
\label{lem:n=m}
The generators $x$ and $x^*$ of the quantum disc $\ffun(D_p)$ satisfy the relations
\begin{equation}
\label{xxsxsx}
x^\mu x^{-\mu}=1+\qQ^p_\mu(X),\quad \mu\in\Z, \quad x^{-n}:=x^{*n},
\quad n\in\N.
\end{equation}
\end{lemma}
\begin{proof}
We proceed by induction on $|\mu|$, and begin by observing that the 
formula~\eqref{xxsxsx} is immediately true for $\mu=0,\pm 1$. 
Suppose the 
formula is satisfied for some $\mu>0$. Then, using 
Lemma~\ref{lem:first-poly} yields
\begin{align}
x^{\mu+1}x^{*(\mu+1)}
&=x(x^\mu x^{*\mu })x^{*}\nonumber\\
&=x\left(
1+\qQ^p_\mu (X)
\right)x^{*}\nonumber\\
&=xx^{*}\left(
1+\qQ^p_\mu (p^{-1}X)
\right)\nonumber\\
&=(1-X)\left(1+\qQ^p_\mu (p^{-1}X)
\right)\nonumber\\
&=1+\qQ^p_{\mu +1}(X).
\end{align}
The proof for $\mu <0$ proceeds in the same way, because the identities for deformed factorials and binomial coefficients are the
same for $0<p<1$ and $p>1$.
\end{proof}

Let us define the  family $\qQ^p_{\mu;\nu}$ of polynomials, for all $\mu,\nu\in\Z$ by the formulae
\begin{equation}
\label{qmndef}
\qQ^p_{\mu;\nu}(Y)=\left\{
\begin{array}{lcl}
0 &\text{if}&\mu\nu\geq 0\\
\qQ^p_\mu(Y) & \text{if}& \mu\nu<0\text{\;and\;}|\mu|\leq |\nu|\\
\qQ^p_{-\nu}(p^{-(\mu+\nu)}Y) &\text{if} & \mu\nu<0\text{\;and\;}|\mu|>|\nu|
\end{array}
\right.\;.
\end{equation}
Now we can generalise Lemma~\ref{lem:n=m}.
\begin{lemma}
\label{lem:m-neq-n}
Let $X:=1-xx^*$, $\mu,\nu\in\Z$, and for 
$\mu<0$ write $x^\mu=x^{*|\mu|}$. Then the generators
of the algebra  $\ffun(D_p)$ satisfy
\begin{equation}
\label{xxsqgen}
x^\mu x^\nu=(1+Q^p_{\mu;\nu}(X))x^{\mu+\nu}.
\end{equation}
\end{lemma}
\begin{proof}
The statement is obvious if $\mu\nu\geq0$, and we now consider the two cases when $\mu\nu<0$. First, for 
$|\mu|\leq|\nu|$ we find 
\[
x^\mu x^\nu=(x^\mu x^{-\mu})x^{\mu+\nu}=(1+\qQ^p_\mu(X))x^{\mu+\nu}=(1+\qQ^p_{\mu;\nu}(X))x^{\mu+\nu}.
\]
Next, for  $|\mu|>|\nu|$  we obtain
\begin{equation*}
x^\mu x^\nu=x^{\mu+\nu}(x^{-\nu}x^\nu)=x^{\nu+\mu} (1+\qQ^p_{-\nu}(X))
=(1+\qQ^p_{-\nu}(p^{-(\nu+\mu)}X))x^{\nu+\mu}
=(1+\qQ^p_{\mu;\nu}(X))x^{\nu+\mu},
\end{equation*}
as needed.
\end{proof}

\subsection{Heegaard quantum lens spaces in terms of generators and relations}

In what follows we will frequently need the following  formula.
Let $x$ and $y$ be two elements in an algebra such that 
$xy=e^{i\phi}yx$, 
where $\phi\in\R$.
Then 
\begin{equation}
\label{chlemma}
x^\mu y^\mu=e^{i\phi\frac{\mu(\mu-1)}{2}}(xy)^\mu,\quad\text{for}\quad
 \mu\in \Z,
\quad x^{-n}:=(x^*)^n,\quad y^{-n}:=(y^*)^n,\quad n\in \N.
\end{equation}

We recall the coaction
of $\ffun(\Z/N\Z)$ on $\ffun(S^3_{pq\theta})$ from 
Subsection~\ref{subsub:lens}. The coaction-invariant subspace of this 
coaction is simply the subspace of elements of
degree divisible by $N$, and it is called a lens space $L^N_{pq\theta}$ 
of type $N$. It follows from \eqref{hbasis} that 
$\ffun(L^N_{pq\theta})$ is spanned as a vector space by the set
\begin{multline}\label{Lbasis}
\{A^ka^\mu b^\nu\;|\;k,\lambda,\mu,\nu\in\Z,\;k\geq 0,\;\mu+\nu =
\lambda N\}
\\ \cup 
\{B^ka^\mu b^\nu\;|\;k,\lambda,\mu,\nu\in\Z,\;k>0,\;\mu+\nu =\lambda N\},
\end{multline} 
where for $\mu,\nu<0$ we have written $b^\mu=b^{*|\mu|}$ and $a^\nu:=a^{*|\nu|}$ for brevity.
Let us also define 
\begin{equation}
\tilde{a}:=a^N,\quad \tilde{b}:=b^N,\quad z:=ab^*.
\end{equation}
It is not difficult to verify that the elements $A$, $B$, $z$, $\tilde{a}$, $\tilde{b}$ satisfy
 the  commutation relations 
\begin{subequations}
\label{lense}
\begin{gather}
A^*=A,\quad B^*=B,\quad AB=0,\quad
Az=pzA,\quad zB=qBz,\label{lensea}\\
z^*z=1-pA-B,\quad zz^*=1-A-qB,\label{lenseb}\\
A\tilde{a}=p^N\tilde{a}A,\;\;
A\tilde{b}=\tilde{b}A,\;\; 
B\tilde{a}=\tilde{a}B,\;\;
 B\tilde{b}=q^N\tilde{b}B,\;\; 
z\tilde{a}=e^{iN2\pi\theta}\tilde{a}z,\;\;
z\tilde{b}^*=e^{-iN2\pi\theta}\tilde{b}^*z,\label{lensec}\\
z\tilde{a}^*-e^{-iN2\pi\theta}\tilde{a}^*z=
e^{-iN(N+1)\pi\theta}(p^N-1)Az^{1-N}\tilde{b}^*,
\label{lensed}\\
zb-e^{iN2\pi\theta}bz=e^{i\pi\theta N(N-1)}q(q^{-N}-1)Bz^{1-N}\tilde{a},\label{lensee}\\
\tilde{a}\tilde{b}=e^{iN^22\pi\theta}\tilde{b}\tilde{a},\quad \tilde{a}\tilde{b}^*=e^{-iN^22\pi\theta}\tilde{b}^*\tilde{a},
\quad \tilde{a}\tilde{b}^*=e^{-i\pi\theta N(N-1)}z^N,\label{lensef}\\
\tilde{a}^*\tilde{a}=1+\qQ^p_{-N}(A),
\quad \tilde{a}\tilde{a}^*=1+\qQ^p_{N}(A),\quad
\tilde{b}^*\tilde{b}=1+\qQ^q_{-N}(B),
\quad \tilde{b}\tilde{b}^*=1+\qQ^q_{N}(B).\label{lenseg}
\end{gather}
\end{subequations}

Here  the polynomials $\qQ^p_\mu$ were defined in Equation~\eqref{qzdef}.
Formulas~\eqref{lensea}-\eqref{lensec} and the first two equations in \eqref{lensef}
are straightforward consequences of Equations~\eqref{heegard}. In order to prove the
 last equality in~\eqref{lensef}, we use~\eqref{chlemma}. Equalities~\eqref{lenseg}
follow immediately from Lemma~\ref{lem:n=m}. 
In order to prove Equation~\eqref{lensed},  we need to do a little work.
First we note that for all $n>0$ we have
\begin{equation}\label{aaminus}
aa^{*n}-a^{*n}a=(p^n-1)Aa^{*(n-1)}.
\end{equation}
Indeed, this formula holds for $n=1$, and  for $n>1$ we can write
\begin{align}
aa^{*n}-a^{*n}a&=(aa^{*})a^{*(n-1)}-a^{*(n-1)}(a^{*}a)\nonumber\\
&=(1-A)a^{*(n-1)}-a^{*(n-1)}(1-pA)\nonumber\\
&=-Aa^{*(n-1)}+pp^{n-1}Aa^{*(n-1)}\nonumber\\
&=(p^n-1)Aa^{*(n-1)}.
\end{align}
Now we are ready to prove Equation~\eqref{lensed}.
Using Equation~\eqref{aaminus} in step $(a)$ and 
Lemma~\ref{lem:m-neq-n} in step~$(b)$, we compute
\begin{align}
z\tilde{a}^*-e^{-iN2\pi\theta}\tilde{a}^*z
&=ab^{-1}a^{-N}-e^{-iN2\pi\theta}a^{-N}ab^{-1}\nonumber\\
&=e^{-iN2\pi\theta}(aa^{-N}-a^{-N}a)b^{-1}\nonumber\\
&\stackrel{(a)}{=}e^{-iN2\pi\theta}(p^N-1)Aa^{1-N}b^{-1}\nonumber\\
&\stackrel{(b)}{=}e^{-iN2\pi\theta}(p^N-1)Aa^{1-N}\left(b^{N-1}b^{-N}-\qQ^q_{N-1;-N}(B)b^{-1}\right).
\end{align}
Next, we use the fact that $A\qQ^q_{N-1;-N}(B)=0$ (due to $AB=0$)
and  commutation relations for $A$ to obtain
\begin{align}
z\tilde{a}^*-e^{-iN2\pi\theta}\tilde{a}^*z&=e^{-iN2\pi\theta}(p^N-1)A(a^{1-N}b^{N-1})b^{-N}\nonumber\\
&=e^{-iN2\pi\theta-i(N-1)^22\pi\theta}(p^N-1)A(a^{N-1}b^{-(N-1)})^*b^{-N}\nonumber\\
&=e^{-iN2\pi\theta-i(N-1)^22\pi\theta+i\pi\theta (N-1)(N-2)}
      (p^N-1)A((ab^*)^{N-1})^*b^{-N}\nonumber\\
&=e^{-iN(N+1)\pi\theta}(p^N-1)Az^{1-N}\tilde{b}^*.
\end{align}
Here the second last equality follows from~\eqref{chlemma}.
The proof of Equation~\eqref{lensee} is similar.

We are now ready for the main claim of this subsection.
\begin{theorem}
\label{thm:basis}
Let  $\mathcal{A}$ be the universal $*$-algebra generated by the elements $\tilde{a}'$, $\tilde{b}'$,
$z'$, $A'$ and~$B'$, and satisfying the same relations~\eqref{lense} as their unprimed counterparts. Then  $\ffun(L^N_{pq\theta})$ and
 $\mathcal{A}$ are isomorphic as $*$-algebras, and the set of vectors 
\begin{equation}
\label{lensbasis}
{\mathcal B}:=\{(A')^k(z')^\mu(\tilde{b}')^\nu\;|\;k>0,\ \mu,\nu\in\Z\}\cup\{(B')^{k}(z')^{\mu}(\tilde{a}')^{\nu}\;|\;k\geq 0,\ \mu,\nu\in\Z\}
\end{equation}
is a basis of ${\mathcal A}$. Here  for $\mu,\nu<0$ we have written $(\tilde{b}')^\mu=(\tilde{b}')^{*|\mu|}$, $(\tilde{a}')^\nu:=(\tilde{a}')^{*|\nu|}$
and $(z')^\mu=(z')^{*|\mu|}$ for brevity.
\end{theorem}

The proof of this theorem will occupy the remainder of this section. Until the final stage of the proof,  we will abuse
notation by dropping the primes on the generators of ${\mathcal A}$. First we will prove some additonal commutation relations.
\begin{lemma}
Let $\mu,\nu\in\Z$, and  for $\mu,\nu<0$  write 
$\tilde{b}^\mu=\tilde{b}^{*|\mu|}$ and 
$\tilde{a}^\nu:=\tilde{a}^{*|\nu|}$. 
Then we have the relations 
\begin{subequations}
\label{multpls}
\begin{gather}
z^\mu z^\nu=(1+\qQ^p_{\mu;\nu}(A)+\qQ^q_{-\mu;-\nu}(B))z^{\mu+\nu},\label{multplsa}\\
\tilde{a}^\mu\tilde{a}^\nu=(1+\qQ^p_{N \mu;N \nu}(A))\tilde{a}^{\mu+\nu},\label{multplsb}\\
\tilde{b}^\mu\tilde{b}^\nu=(1+\qQ^q_{N \mu;N \nu}(B))\tilde{b}^{\mu+\nu},\label{multplsc}\\
A\tilde{a}^\nu=e^{-i\pi\theta N \nu(N \nu-1)}Az^{N \nu}\tilde{b}^\nu,\label{multplse}\\
Az^\nu\tilde{b}^\mu=e^{iN2\pi\theta \mu\nu}A\tilde{b}^\mu z^\nu,\label{multplsi}
\end{gather}
where the polynomials $\qQ^p_{\mu;\nu}$ were defined in~\eqref{qmndef}.
\end{subequations}
\end{lemma}
\begin{proof}
We prove  each of the Equations~\eqref{multpls} separately.
For Equation~\eqref{multplsa}, we first  prove, by induction, the simpler result
\begin{equation}
\label{multplshelp}
z^nz^{*n}=1+\qQ^p_n(A)+\qQ^q_{-n}(B),\qquad z^{*n}z^n=1+\qQ^p_{-n}(A)+\qQ^p_n(B).
\end{equation}
Equations~\eqref{multplshelp} is clearly satisfied for $n=0, 1$ by Equations~\eqref{lenseb} and \eqref{qzdef}.
We will prove the first equality, the second being proved similarly. So suppose that $n>0$. Then
\begin{align}
z^{n+1}z^{*(n+1)}&=z(z^nz^{*n})z^{*}\nonumber\\
&=z\left(1+\qQ^p_n(A)+\qQ^q_{-n}(B)\right)z^{*}\nonumber\\
&=zz^{*}\left(1+\qQ^p_n(p^{-1}A)+\qQ^q_{-n}(qB)\right)\nonumber\\
&=(1-A-qB)\left(1+\qQ^p_n(p^{-1}A)+\qQ^q_{-n}(qB)\right)\nonumber\\
&=1+\left((1-A)\qQ^p_n(p^{-1}A)-A\right)+\left((1-qB)\qQ^q_{-n}(qB)-qB\right)\nonumber\\
&=1+\qQ^p_{n+1}(A)+\qQ^q_{-(n+1)}(B),
\end{align}
where in the last equality we used the recursive relations~\eqref{qdefrec} and \eqref{qminrec}.
Now Equation~\eqref{multplsa} clearly holds when $\mu\nu\geq 0$.
When $\mu\nu<0$ there are four cases, and we will show how the proof works in a single instance. Suppose then, 
that $|\mu|\leq |\nu|$ and $\nu<0$. Then using~\eqref{multplshelp}
and \eqref{qmndef} we obtain
\begin{align}
z^{|\mu|} z^{*|\nu|}&=(z^{|\mu|} z^{*|\mu|})z^{*(|\nu|-|\mu|)}\nonumber\\
&=(1+\qQ^p_\mu(A)+\qQ^q_{-\mu}(B))z^{*(|\nu|-|\mu|)}\nonumber\\
&=(1+\qQ^p_{\mu;\nu}(A)+\qQ^q_{-\mu;-\nu}(B))z^{*(|\nu|-|\mu|)}.
\end{align}

Equations~\eqref{multplsb} and \eqref{multplsc} are proved in the same way, and we just prove Equation~\eqref{multplsb}. 
First we prove the formulae
\begin{equation}
\label{multbh}
\tilde{a}^n\tilde{a}^{*n}=1+\qQ^p_{N n}(A),\qquad \tilde{a}^{*n}\tilde{a}^{n}=1+\qQ^p_{-N n}(A),
\end{equation}
by induction. By \eqref{lenseg}, these formulae are true for $n=0, 1$.
For $n>0$ we use
formula~\eqref{qmneq} and Equation~\eqref{qzdef} to find
\begin{align}
\tilde{a}^{n+1}\tilde{a}^{*(n+1)}&=\tilde{a}\tilde{a}^n\tilde{a}^{*n}\tilde{a}^{*}\nonumber\\
&=\tilde{a}(1+\qQ^p_{N n}(A))\tilde{a}^{*}\nonumber\\
&=\tilde{a}\tilde{a}^{*}(1+\qQ^p_{N n}(p^{-N}A))\nonumber\\
&=(1+\qQ^p_N(A))(1+\qQ^p_{N n}(p^{-N}A))\nonumber\\
&=1+\qQ^p_{N(n+1)}(A).
\end{align}
Similarly 
\[
\tilde{a}^{*(n+1)}\tilde{a}^{n+1}=\tilde{a}^*(1+\qQ^p_{-nN}(A))\tilde{a}=(1+\qQ^p_{-N}(A))(1+\qQ^p_{-nN}(p^NA))=1+\qQ^p_{-N(n+1)}(A).
\]
Now the full formula~\eqref{multplsb} can be easily proven 
using \eqref{multbh} and \eqref{qmndef} by
applying similar methods to those used to prove
Equation~\eqref{multplsa}.


To prove Equation~\eqref{multplse}, first note that, as a direct
consequence of \eqref{chlemma} and \eqref{lensef}, we obtain 
\begin{equation}\label{multpleq}
\tilde{a}^\mu\tilde{b}^{-\mu}=e^{-i\pi\theta N \mu(N \mu -1)}z^{N \mu},
\quad\tilde{a}^{-n}:=\tilde{a}^{*n},
\quad\tilde{b}^{-n}:=\tilde{b}^{*n},
\quad\mu,n\in\Z,\quad n>0.
\end{equation}
Then Equation~\eqref{multplsc} and Equation~\eqref{multpleq} yield
\begin{equation}
A\tilde{a}^\mu =A\tilde{a}^\mu \tilde{b}^{-\mu }\tilde{b}^\mu -A\tilde{a}^\mu \qQ^{q}_{-\mu N;\mu N}(B)
=A(\tilde{a}^\mu \tilde{b}^{-\mu })\tilde{b}^\mu -0=
e^{-i\pi\theta N \mu (N \mu -1)}Az^{N \mu }\tilde{b}^\mu .
\end{equation}

Finally, Equation~\eqref{multplsi} follows directly from the commutation relations 
between $\tilde{b}^{\pm 1}$ and $z^{\pm 1}$ 
(Equations~\eqref{lensec}, \eqref{lensee}) and the fact that the additional term which might appear as a side 
effect of commuting $\tilde{b}^{\pm 1}$ with $z^{\pm 1}$ is proportional to $B$.
\end{proof}

Let $\mathcal{V}$ be the linear subspace of $\mathcal{A}$ 
spanned by ${\mathcal B}$.
Our aim is to show first that the generators of ${\mathcal A}$ belong
to ${\mathcal V}$. Then, using some additional commutation
relations, we will show that ${\mathcal V}$ is closed under 
multiplication, and hence 
equal to ${\mathcal A}$. Finally, we will argue that 
${\mathcal A}$  is isomorphic to $\ffun(L^N_{pq\theta})$.
First we prove that  $\tilde{b},\,\tilde{b}^*\in {\mathcal V}$.
Equations~\eqref{lense} allow us to write
\begin{equation}
\tilde{b}=\tilde{b}(\tilde{a}^*\tilde{a}-\qQ^p_{-N}(A))
=e^{i\pi\theta N(N-1)}z^{*N}\tilde{a}-\qQ^p_{-N}(A)\tilde{b}.
\end{equation}
Similarily
$\tilde{b}^*=\tilde{b}^*(\tilde{a}\tilde{a}^*-\qQ^p_N(A))
=e^{i\pi\theta N(N+1)}z^N\tilde{a}^*
-\qQ^p_N(A)\tilde{b}^*$,
which completes the argument, since $\qQ^p_N(A)$ has no constant term.

The previous argument, along with the definition of ${\mathcal V}$, shows that all the generators $\tilde{a}$, $\tilde{b}$, $z$, $A$, $B$,
are contained in ${\mathcal V}$. Thus to show that ${\mathcal V}=
{\mathcal A}$ we just need to prove that ${\mathcal V}$ is closed 
under multiplication.
To this end, let us denote for brevity the following linear subspaces of $\mathcal{V}$:
\begin{subequations}
\begin{gather}
\mathcal{V}_A:=\Span\{A^kz^\mu\tilde{b}^\nu\;|\;k,\mu,\nu\in\Z,\;k>0\},\\
\mathcal{V}_0:=\Span\{z^\mu\tilde{a}^\nu\;|\;\mu,\nu\in\Z\},\quad
\mathcal{V}_B:=\Span\{B^kz^\mu\tilde{a}^\nu\;|\;k,\mu,\nu\in\Z,\;k>0\},\\
\mathcal{W}:=\Span(\mathcal{V}_0\cup\mathcal{V}_B).
\end{gather}
\end{subequations}
Here  for $\mu,\nu<0$ we have written $\tilde{b}^\mu=\tilde{b}^{*|\mu|}$, $\tilde{a}^\nu:=\tilde{a}^{*|\nu|}$
and $z^\mu=z^{*|\mu|}$ for brevity. The relation between these subspaces and ${\mathcal V}$ is
\[
{\mathcal V}=\Span({\mathcal V}_A\cup {\mathcal W})=\Span({\mathcal V}_A\cup {\mathcal V}_0\cup {\mathcal V}_B).
\]

\begin{lemma}\label{muldlem}
For all $\nu\in \Z$, we have the inclusions
\begin{subequations}
\label{muld}
\begin{gather}
\mathcal{V}_A\tilde{b}^\nu\subseteq\mathcal{V}_A,\quad \tilde{b}^\nu\mathcal{V}_A\subseteq\mathcal{V}_A,\label{muldb}\\
\mathcal{V}_Az^\nu\subseteq\mathcal{V}_A,\quad z^\nu\mathcal{V}_A\subseteq\mathcal{V}_A,\label{muldz}\\
\mathcal{V}_A\tilde{a}^\nu\subseteq\mathcal{V}_A,\quad \tilde{a}^\nu\mathcal{V}_A\subseteq \mathcal{V}_A.\label{mulda}
\end{gather}
\end{subequations}
\end{lemma}
\begin{proof}
It is enough to consider an arbitrary vector from the set spanning $\mathcal{V}_A$, namely
$A^kz^\sigma\tilde{b}^\tau$, where  $k,\sigma,\tau\in\Z$ and $k>0$.
It follows from Equation~\eqref{multplsc} that
$A^kz^\sigma\tilde{b}^\tau\tilde{b}^\nu=A^kz^\sigma\tilde{b}^{\tau+\nu}\in\mathcal{V}_A$. Similarly, using equations \eqref{multplsi}
and \eqref{multplsc}, we obtain
\begin{equation}
\tilde{b}^\nu A^kz^\sigma\tilde{b}^\tau=A^k\tilde{b}^\nu z^\sigma\tilde{b}^\tau=e^{-iN2\pi\theta \nu\sigma}A^kz^\sigma\tilde{b}^\nu\tilde{b}^\tau
=e^{-iN2\pi\theta \nu\sigma}A^kz^\sigma\tilde{b}^{\nu+\tau}\in\mathcal{V}_A.
\end{equation}
Next, using equations \eqref{multplsa} and \eqref{multplsi},
we obtain
\begin{equation}
A^kz^\sigma\tilde{b}^\tau z^\nu=e^{-iN2\pi\theta \tau\nu}A^kz^\sigma z^\nu\tilde{b}^\tau
=e^{-iN2\pi\theta \tau\nu}A^k(1+\qQ^p_{\sigma;\nu}(A))z^{\sigma+\nu}\tilde{b}^\tau
\in\mathcal{V}_A.
\end{equation}
Similarly, Equation~\eqref{multplsa} yields
\begin{equation}
z^\nu A^kz^\sigma\tilde{b}^\tau=p^{-\nu k}A^kz^\nu z^\sigma\tilde{b}^\tau
=p^{-\nu k}A^k(1+\qQ^p_{\nu;\sigma}(A))z^{\nu+\sigma}\tilde{b}^\tau\in\mathcal{V}_A.	
\end{equation}
The final inclusions follows immediately from Equation~\eqref{muldz} and Equation~\eqref{muldb} using
Equation~\eqref{multplse}.
\end{proof}

\begin{lemma}
For all $\mu,\nu\in\Z$, we have the commutation relation
\begin{equation}\label{azch}
\tilde{a}^\nu z^\mu- e^{-iN2\pi\theta \mu\nu}z^\mu\tilde{a}^\nu\in \mathcal{V}_A.
\end{equation}
Here  for $\mu,\nu<0$ we have written  $\tilde{a}^\nu:=\tilde{a}^{*|\nu|}$
and $z^\mu=z^{*|\mu|}$ for brevity.
\end{lemma}

\begin{proof}
The cases $\mu$ and $\nu$ both positive or both negative follow immediately from Equation~\eqref{lensec}, even replacing ${\mathcal V}_A$
by the zero subspace. For $\mu>0$ and $\nu<0$, the result follows from Equation~\eqref{lensed}. The final case follows by taking
the adjoint of Equation~\eqref{lensed}, manipulating the result using Equations~\eqref{lensea}, \eqref{lensec}, and then finally applying Equation~\eqref{multplsi}.
\end{proof}

\begin{lemma}
The vector subspace $\mathcal{V}\subseteq\mathcal{A}$ is closed under multiplication.
\end{lemma}
\begin{proof}
It is enough to consider products of basis vectors. First we note that because the only cost of commuting
$A$ and $B$ through any other generator of $\mathcal{V}$ is the appearing of central coefficients and $AB=BA=0$, we can easily conclude that $\mathcal{V}_A\mathcal{V}_B=\mathcal{V}_B\mathcal{V}_A=0$. 
Next,
from Lemma~\ref{muldlem} one immediately concludes that
$\mathcal{V}_A\cdot\mathcal{V}_A\subseteq\mathcal{V}_A$,
$\mathcal{V}_A\cdot\mathcal{V}_0\subseteq\mathcal{V}_A$ and
$\mathcal{V}_0\cdot\mathcal{V}_A\subseteq\mathcal{V}_A$.

Furthermore, using  \eqref{azch}, \eqref{muld}, \eqref{multplsa}, 
\eqref{multplsb}, one can conclude that,
for all $\mu,\nu,\sigma,\tau\in\Z$,
\begin{align}
z^\mu\tilde{a}^\nu z^\sigma\tilde{a}^\tau &\in 
e^{-iN2\pi\theta \nu\sigma}z^\mu z^\sigma\tilde{a}^\nu\tilde{a}^\tau+z^\mu\mathcal{V}_A\tilde{a}^\tau
\subseteq e^{-iN2\pi\theta \nu\sigma}z^\mu z^\sigma\tilde{a}^\nu\tilde{a}^\tau+\mathcal{V}_A\nonumber\\
&=e^{-iN2\pi\theta \nu\sigma}(1+\qQ^p_{\mu;\sigma}(A)+\qQ^q_{-\mu;-\sigma}(B))z^{\mu+\sigma}(1+\qQ^p_{N \nu;N \tau}(A))\tilde{a}^{\nu+\tau}
+\mathcal{V}_A\nonumber\\
&=e^{-iN2\pi\theta \nu\sigma}(1+\qQ^{\nu,\tau}_{\mu,\sigma}(A)+\qQ^q_{-\mu;-\sigma}(B))z^{\mu+\sigma}\tilde{a}^{\nu+\tau}+\mathcal{V}_A,
\end{align}
where we have denoted $\qQ^{\nu,\tau}_{\mu,\sigma}(A)=\qQ^p_{\mu;\sigma}(A)+\qQ^p_{N \nu;N \tau}(p^{-\mu-\sigma}A)(1+
\qQ^p_{\mu;\sigma}(A))$ for brevity. Observe that due to Equation~\eqref{multplse}
we have
$\qQ^{\nu,\tau}_{\mu,\sigma}(A)z^{\mu+\sigma}\tilde{a}^{\nu+\tau}\in\mathcal{V}_A$, so that 
$z^\mu\tilde{a}^\nu z^\sigma\tilde{a}^\tau\in\mathcal{V}$
for all $\mu,\nu,\sigma,\tau\in\Z$. It follows immediately that also 
$\mathcal{W}\cdot\mathcal{W}\subseteq\mathcal{V}$, which completes
the proof.
\end{proof}

Summarising, we conclude that $\mathcal{V}=\mathcal{A}$ because $\mathcal{V}$ contains generators of $\mathcal{A}$ 
and is closed under multiplication. Hence the vectors in the set ${\mathcal B}$, Equation~\eqref{lensbasis}, span $\mathcal{A}$.
This proves half of Theorem~\ref{thm:basis}, and we now complete the proof. 
To this end, we take the natural $*$-homomorphism $f:\mathcal{A}\rightarrow \ffun(L^N_{pq\theta})$ defined on generators of $\mathcal{A}$ by
\begin{equation}
A'\mapsto A:=1-aa^*,\quad B'\mapsto B:=1-bb^*,\quad z'\mapsto ab^*,\quad \tilde{a}'\mapsto a^N,\quad
\tilde{b}'\mapsto b^N.
\end{equation}
It is enough to prove that this $*$-homomorphism is a linear bijection.


Before doing so, we note that it follows from
 \eqref{heegard} and~\eqref{chlemma} that 
\begin{equation}
\label{fzet}
f((z')^\mu)=e^{i\pi\theta\mu(\mu-1)}a^\mu b^{-\mu},\quad \mu\in\Z.
\end{equation}
Also, it follows from Equation~\eqref{xxsqgen} and relations~\eqref{heegard}, that the
values of $f$ on linear generators~\eqref{lensbasis} of $\mathcal{A}$ are given by
\begin{subequations}
\label{fongens}
\begin{gather}
f((A')^k(z')^\mu(\tilde{b}')^\nu)=e^{i\pi\theta\mu(\mu-1)}A^ka^\mu b^{N \nu-\mu}, \label{fongensa}\\
f((B')^k(z')^\mu(\tilde{a}')^\nu)=e^{i2\pi\theta\left(\frac{\mu(\mu-1)}{2}+N \mu\nu\right)}B^ka^{\mu+N \nu}b^{-\mu},\label{fongensb}\\
f((z')^\mu(\tilde{a}')^\nu)=e^{i2\pi\theta\left(\frac{\mu(\mu-1)}{2}+N \mu\nu\right)}(1+Q_{\mu;N \nu}^p(A))a^{\mu+N \nu}b^{-\mu},\label{fongensc}
\end{gather}
\end{subequations}
for all $k,\mu,\nu\in\Z$, $k>0$.

First we show that the homomorphism $f$ is surjective. It is enough to prove that an arbitrary vector from basis~\eqref{Lbasis}
is in the image of $f$. For all $k,\mu,\nu\in\Z$, $k>0$,   we have   
\begin{equation}\label{finva}
A^ka^\mu b^{N \nu -\mu}=f\left(e^{-i\pi\theta\mu(\mu-1)}(A')^k(z')^\mu(\tilde{b}')^\nu\right).
\end{equation}
On the other hand, \eqref{fongensb} implies that for all 
$k,\mu,\nu\in\Z$, $k>0$,  we have 
\begin{equation}\label{finvb}
B^ka^\mu b^{N \nu -\mu}=f\left(
e^{-i2\pi\theta\left(\frac{(-N \nu+\mu)(-N \nu+\mu-1)}{2}+N(-N \nu+\mu)\nu\right)}
(B')^k(z')^{-N \nu+\mu}(\tilde{a}')^\nu
\right).
\end{equation}
Finally,  it follows from the same computation which led
to Equation~\eqref{finvb} that for all $\mu,\nu\in \Z$
\begin{equation}
f\left(e^{-i2\pi\theta\left(\frac{(-N \nu+\mu)(-N \nu+\mu-1)}{2}+N(-N \nu+\mu)\nu\right)}
(z')^{-N \nu+\mu}(\tilde{a}')^\nu\right)
=(1+Q_{-N \nu+\mu;N \nu}^p(A))a^\mu b^{N \nu-\mu}.
\end{equation}
Hence, using Equation~\eqref{finva} we find that for all $\mu,\nu\in \Z$
\begin{align}
a^\mu b^{N \nu-\mu}\label{finvc}
&=f\!\left(\!e^{-i2\pi\theta\left(\frac{(-N \nu+\mu)(-N \nu+\mu-1)}{2}+N(-N \nu+\mu)\nu\right)}
(z')^{-N \nu+\mu}(\tilde{a}')^\nu\right)
\\
&\phantom{=}\;-f\left(e^{-i\pi\theta\mu(\mu-1)}Q_{-N \nu+\mu;N \nu}^p(A')(z')^\mu(\tilde{b}')^\nu\!\right).\nonumber
\end{align}

Next, to show that the homomorphism $f$ is injective, we suppose that $f(v)=0$ for some 
$v\in\mathcal{A}$. Since the set ${\mathcal B}$ spans ${\mathcal A}$, we can write $v$ as a linear combination 
\begin{equation*}
v=\sum_{k>0;\;\mu,\nu\in\Z}\alpha_{k\mu\nu}(A')^k(z')^\mu(\tilde{b}')^\nu+\sum_{k'>0;\;\mu',\nu'\in\Z}\beta_{k'\mu'\nu'}(B')^{k'}(z')^{\mu'}(\tilde{a}')^{\nu'}
+\sum_{\mu'',\nu''\in\Z}\gamma_{\mu''\nu''}(z')^{\mu''}(\tilde{a}')^{\nu''}.
\end{equation*}
Using equations~\eqref{fongens}, we can  explicitly compute $f(v)$ to be
\begin{align*}
&f\left(
\sum_{k>0;\;\mu,\nu\in\Z}\alpha_{k\mu\nu}(A')^k(z')^\mu(\tilde{b}')^\nu+\sum_{k'>0;\;\mu',\nu'\in\Z}\beta_{k'\mu'\nu'}(B')^{k'}(z')^{\mu'}(\tilde{a}')^{\nu'}
+\sum_{\mu'',\nu''\in\Z}\gamma_{\mu''\nu''}(z')^{\mu''}(\tilde{a}')^{\nu''}
\right)\nonumber\\
&=\sum_{k>0;\;\mu,\nu\in\Z}\alpha_{k\mu\nu}e^{i\pi\theta\mu(\mu-1)}A^ka^\mu b^{N \nu-\mu}
+\sum_{k'>0;\;\mu',\nu'\in\Z}\beta_{k'\mu'\nu'}e^{i2\pi\theta\left(\frac{\mu'(\mu'-1)}{2}-N \mu'\nu'\right)}B^{k'}a^{\mu'+N \nu'}b^{-\mu'}\nonumber\\
&\phantom{=}\;+\sum_{\mu'',\nu''\in\Z}\gamma_{\mu''\nu''}e^{i2\pi\theta\left(\frac{\mu''(\mu''-1)}{2}-N \mu''\nu''\right)}
(1+Q^p_{\mu'';N \nu''}(A))a^{\mu''+N \nu''}b^{-\mu''}.
\end{align*}
Since the set of vectors~\eqref{Lbasis} is linearly independent, it follows immediately that $f(v)=0$ implies
that $\beta_{k'\mu'\nu'}=0$, for all $k',\mu',\nu'\in\Z$, $k'>0$. 
Now considering the terms 
$\gamma_{\mu''\nu''}a^{\mu''+N \nu''}b^{-\mu''}$ 
in the last sum, we see that  $\gamma_{\mu''\nu''}=0$, for all 
$\mu'',\nu''\in\Z$. Then also
$\alpha_{k\mu\nu}=0$, for all $k,\mu,\nu\in\Z$, $k>0$. Hence $v=0$, so 
that $f$ is injective. Finally,
note that this also proves that the set of vectors~\eqref{lensbasis} is  
linearly independent.

\subsection{Non-cleftness  of the Heegaard $\mathcal{O}(\mathbb{Z}/N\mathbb{Z})$-comodule algebras}

Let us begin by showing that $\ffun(S^3_{pq})$ is a
principal comodule algebra. As explained in the preliminaries,
to this end it suffices to construct a strong connection.
It turns out that a simple modification of the formulae for a
strong connection given in \cite[(4.4)--(4.6)]{hms06a} yields a 
strong connection in our case.
We define the linear map 
\[
\ell:\ffun(\Z/N\Z)\ni h\longmapsto
\ell(h)=:h^{\langle 1\rangle}\ot h^{\langle 2\rangle}
\in\ffun(S^3_{pq\theta})\ot\ffun(S^3_{pq\theta})\quad
\]
 (summation understood in 
$h^{\langle 1\rangle}\ot h^{\langle 2\rangle}$) 
by setting its values on the basis elements
$\tilde{u}^n$, $n=0,\ldots,N-1$, to be 
\[
\ell(1)=1\ot 1,\quad
\ell(\tilde{u})=a^*\ot a + p^{-1}b^*A\ot b,\quad
\label{algstrong}
\ell(\tilde{u}^n)=\tilde{u}^{\langle 1\rangle}
\ell(\tilde{u}^{n-1})\tilde{u}^{\langle 2\rangle},\quad
n=2,\ldots,N-1.
\]
Proving that the above defined map is a strong connection is
almost identical to an argument provided in~\cite{hms06a}.

The method we use to show
the non-existence of a cleaving map, is to prove that there are not enough invertibles in the comodule algebra of the quantum sphere
to accomodate the range of such a map. 
\begin{theorem}\label{noninv}
The only invertible elements in the algebra of polynomial functions on the Heegaard quantum sphere $\ffun(S^3_{pq\theta})$
are non-zero multiples of the identity.
\end{theorem}
\begin{proof}
Our proof will follow the general idea and structure 
of the proof of non-existence of non-trivial invertible elements in 
another noncommutative deformation of the polynomial algebra of $S^3$,
namely the Hopf algebra $\ffun(SL_q(2))$~\cite{hm99}.
Here we use the basis~\eqref{hbasis} to present 
each element in $\ffun(S^3_{pq\theta})$
as a linear combination of  monomials $a^\mu b^\nu$, $\mu,\nu\in\Z$,
 with coefficients in the polynomial algebra
generated by $A$ and $B$.  The crux of the proof is that  $a$ and
$b$ are invertible  up to  polynomials in $A$ and $B$.

For the duration of this proof, 
for $\Z\ni \mu,\nu<0$ we will write $b^\mu=b^{*|\mu|}$ and $a^\nu:=a^{*|\nu|}$ for brevity.
Recall that any element   $r\in \ffun(S^3_{pq\theta})$ can be expanded using basis~\eqref{hbasis} as
\begin{equation}\label{Cmnr}
r=\sum_{\stackrel{\mu,\nu\in\Z}{\mbox{\scriptsize{finite}}}}C^r_{\mu;\nu}(A,B)a^\mu b^\nu,
\end{equation}
where $\{C^r_{\mu;\nu}\}_{\mu,\nu\in\Z}$ is a family of complex polynomials in two variables  without mixed monomials.
Here no mixed monomials  means that
it can be written as 
\begin{equation}
\label{Cmnsplit}
C^r_{\mu;\nu}(X,Y)=\gamma^r_{\mu;\nu}+\alpha^r_{\mu;\nu}(X)+\beta^r_{\mu;\nu}(Y),
\end{equation}
where $\gamma^r_{\mu;\nu}\in\C$ and
 $\alpha^r_{\mu;\nu}$ and $\beta^r_{\mu;\nu}$ are polynomials such that
$\alpha^r_{\mu;\nu}(0)=0$ and $\beta^r_{\mu;\nu}(0)=0$.
Since the family of vectors~\eqref{hbasis} is a linear basis, it follows 
that the
polynomials $C^r_{\mu;\nu}$ are uniquely determined by $r$. 

We endow  $\Z\times\Z$
with the  lexicographical order, i.e.,
\begin{equation}\label{lexico}
(\mu,\nu)\leq (\mu',\nu')\quad :=\quad \mu< \mu' \vee (\mu=\mu'\wedge \nu\leq \nu').
\end{equation} 
This order is linear (total) and satisfies
\begin{equation}
\label{ordmin}
(\mu,\nu)\leq (\mu',\nu')\quad \Longleftrightarrow\quad (-\mu',-\nu')\leq (-\mu,-\nu).
\end{equation}
The latter justifies introducing a notation $-(\mu,\nu)=(-\mu,-\nu)$.
Next, we define a $\Z^2$-grading $\deg:\ffun(S^3_{pq\theta})\to \Z^2$ 
by declaring
\begin{equation}
\label{gradzz}
\dgz(C(A,B)a^\mu b^\nu):=(\mu,\nu),\quad \text{for all\ }\mu,\nu\in\Z\text{\ and\ }C(A,B)\neq 0,
\end{equation} 
where $C$ is a polynomial in two variables with coefficients in $\C$.

We will also need to divide $\ffun(S^3_{pq\theta})$ into the following linear subspaces:
\begin{align}
\mathcal{X}&=\Span\{A^ka^\mu b^\nu\;|\;k\geq 0,\;\mu,\nu\in\Z\},\quad
&\mathcal{X}_1&=\Span\{A^ka^\mu b^\nu \;|\;k> 0,\;\mu,\nu\in\Z\},\nonumber\\
\mathcal{Y}&=\Span\{B^ka^\mu b^\nu \;|\;k\geq 0,\;\mu,\nu\in\Z\},\quad
&\mathcal{Y}_1&=\Span\{B^ka^\mu b^\nu\;|\;k> 0,\;\mu,\nu\in\Z\}.\label{hsubs}
\end{align}
Note that
\begin{equation}
\label{splitprop}
\mathcal{X}\cap\mathcal{Y}_1=\mathcal{X}_1\cap\mathcal{Y}=\{0\},\quad
\mathcal{X}\oplus\mathcal{Y}_1=\mathcal{X}_1\oplus\mathcal{Y}=\ffun(S^3_{pq\theta}),
\end{equation}
i.e., we have two splittings of $\ffun(S^3_{pq\theta})$ into a direct sum of subspaces.
The expansion~\eqref{Cmnr} and formula~\eqref{Cmnsplit} provide a way to split
any element $r\in\ffun(S^3_{pq\theta})$ into a sum of vectors from these subspaces. For instance, we can split $r$ into vectors
\begin{equation}
\sum_{\mu,\nu\in\Z}(\gamma^r_{\mu;\nu}+\alpha^r_{\mu;\nu}(A))a^\mu b^\nu\in\mathcal{X}\quad\mbox{and}\quad
\sum_{\mu,\nu\in\Z}\beta^r_{\mu;\nu}(B)a^\mu b^\nu\in\mathcal{Y}_1.
\end{equation}

Using commutation relations~\eqref{heegard} and Equation~\eqref{xxsqgen},
we can write the expansion of a product of elements $r$ and $s$ of $\ffun(S^3_{pq\theta})$ as
\begin{align}
rs&=\left(\sum_{\mu,\nu\in\Z}C^r_{\mu;\nu}(A,B)a^\mu b^\nu\right)
\left(\sum_{\mu',\nu'\in\Z}C^s_{\mu';\nu'}(A,B)a^{\mu'}b^{\nu'}\right)\nonumber\\
&=\sum_{\mu,\nu,\mu',\nu'\in\Z}e^{i2\pi\theta \nu\mu'}C^r_{\mu;\nu}(A,B)C^s_{\mu';\nu'}(p^{-\mu}A,q^{-\nu}B)(a^\mu a^{\mu'})(b^\nu b^{\nu'})\nonumber\\
&=\sum_{\mu,\nu,\mu',\nu'\in\Z}e^{i2\pi\theta \nu\mu'}C^r_{\mu;\nu}(A,B)C^s_{\mu';\nu'}(p^{-\mu}A,q^{-\nu}B)(1+\qQ^p_{\mu;\mu'}(A))
(1+\qQ^q_{\nu;\nu'}(B))a^{\mu+\mu'}b^{\nu+\nu'}.
\end{align}
Hence we obtain
\begin{align}
\label{Cmnmult}
& C^{rs}_{\mu;\nu}(A,B)\\
&=\!\!\!\sum_{\mu',\nu'\in\Z}\!\!\!e^{i2\pi\theta \nu'(\mu-\mu')}C^r_{\mu';\nu'}(A,B)C^s_{\mu-\mu';\nu-\nu'}(p^{-\mu'}A,q^{-\nu'}B)
\left(1+\qQ^p_{\mu';\mu-\mu'}(A)+\qQ^q_{\nu';\nu-\nu'}(B)\right)\nonumber.
\end{align}
It follows immediately that 
\begin{equation}
\label{multsplit}
\mathcal{X}\mathcal{Y}_1\subseteq\mathcal{Y}_1,\quad
\mathcal{Y}_1\mathcal{X}\subseteq\mathcal{Y}_1,\quad
\mathcal{Y}\mathcal{X}_1\subseteq\mathcal{X}_1,\quad
\mathcal{X}_1\mathcal{Y}\subseteq\mathcal{X}_1.
\end{equation}

Furthermore, writing
\begin{equation}\label{58}
r=\sum_{\stackrel{\mu,\nu\in\Z}{\mbox{\scriptsize{finite}}}}C^r_{\mu;\nu}(A,B)a^\mu b^\nu
=\sum_{\stackrel{\mu,\nu\in\Z}{\mbox{\scriptsize{finite}}}}\left(\gamma^r_{\mu;\nu}+\alpha^r_{\mu;\nu}(A)+\beta^r_{\mu;\nu}(B)\right)a^\mu b^\nu
\end{equation}
we observe that if $r$ is invertible and either all $\gamma^r_{\mu;\nu}$'s or all $\gamma^{r^{-1}}_{\mu;\nu}$'s are
 zero, then the expansion of $rr^{-1}$
in terms of basis~\eqref{hbasis} would contain only vectors from 
$\mathcal{X}_1$ or $\mathcal{Y}_1$. However, this is impossible because 
$1\notin \mathcal{X}_1\oplus \mathcal{Y}_1$. Hence, if $r$ is invertible,
 then at least one of the $\gamma^r_{\mu;\nu}$'s 
(and also one of the $\gamma^{r^{-1}}_{\mu;\nu}$'s) is non-zero. 
This observation allows us to define  maps on the invertible elements of 
$\ffun(S^3_{pq\theta})$ by
\begin{subequations}
\label{maxmins}
\begin{align}
\mxda:&\ffun(S^3_{pq\theta})\longrightarrow \Z\times\Z,&\quad r&\longmapsto \max\{(\mu,\nu)\in\Z\times\Z\;|\;
\gamma^r_{\mu;\nu}+\alpha^r_{\mu;\nu}(A)\neq 0\},\\
\mxdb:&\ffun(S^3_{pq\theta})\longrightarrow \Z\times\Z,&\quad r&\longmapsto \max\{(\mu,\nu)\in\Z\times\Z\;|\;
\gamma^r_{\mu;\nu}+\beta^r_{\mu;\nu}(B)\neq 0\},\\
\mnda:&\ffun(S^3_{pq\theta})\longrightarrow \Z\times\Z,&\quad r&\longmapsto \min\{(\mu,\nu)\in\Z\times\Z\;|\;
\gamma^r_{\mu;\nu}+\alpha^r_{\mu;\nu}(A)\neq 0\},\\
\mndb:&\ffun(S^3_{pq\theta})\longrightarrow \Z\times\Z,&\quad r&\longmapsto \min\{(\mu,\nu)\in\Z\times\Z\;|\;
\gamma^r_{\mu;\nu}+\beta^r_{\mu;\nu}(B)\neq 0\}.
\end{align}
\end{subequations}

Next, let $r,s\in\ffun(S^3_{pq\theta})$ be invertible, and let 
$(\mu,\nu)=\mxda(r)$ and $(\mu',\nu')=\mxda(s)$. 
Then, by Equation~\eqref{Cmnmult},
\begin{equation}
\gamma^{rs}_{\mu+\mu';\nu+\nu'}+\alpha^{rs}_{\mu+\mu';\nu+\nu'}(A)=e^{i2\pi\theta \nu\mu'}(\gamma^{r}_{\mu;\nu}+\alpha^{r}_{\mu;\nu}(A))
(\gamma^{s}_{\mu';\nu'}+\alpha^{s}_{\mu'\nu'}(p^{-\mu}A))(1+\qQ^p_{\mu;\mu'}(A)).
\end{equation}
The factors on the right-hand side are non-zero by the definition
of $(\mu,\nu)$ and $(\mu',\nu')$ and because the algebra generated by 
$A$ does not contain zero-divisors.
It follows that $\gamma^{rs}_{\mu+\mu';\nu+\nu'}+\alpha^{rs}_{\mu+\mu';\nu+\nu'}(A)\neq 0$.
Therefore, for all invertible $r$ and $s$  we have
\begin{subequations}
\label{maxminsum}
\begin{equation}
\mxda(rs)=\mxda(r)+\mxda(s),
\end{equation}
where the addition of
 pairs of integers is done componentwise. 
Similarily, we prove that for all invertible $r$ and $s$ we have
\begin{align}
\mxdb(rs)&=\mxdb(r)+\mxdb(s),\\
\mnda(rs)&=\mnda(r)+\mnda(s),\\
\mndb(rs)&=\mndb(r)+\mndb(s).
\end{align}
\end{subequations}

Suppose now that $r\in \ffun(S^3_{pq\theta})$ is invertible. Then $\mxda(rr^{-1})=\mxda(1)=(0,0)$, and similarily
$\mxdb(rr^{-1})=\mndb(rr^{-1})=\mnda(rr^{-1})=(0,0)$. Hence Equations~\eqref{maxminsum} imply that
\begin{gather}
\label{minmaxmin}
\mxda(r^{-1})=-\mxda(r),\quad
\mxdb(r^{-1})=-\mxdb(r),\\
\mnda(r^{-1})=-\mnda(r),\quad
\mndb(r^{-1})=-\mndb(r).\nonumber
\end{gather}
In particular, starting with an obvious property that $\mnda(r^{-1})\leq \mxda(r^{-1})$ and substituting Equation~\eqref{minmaxmin} into it
yields that $-\mnda(r)\leq-\mxda(r)$. Hence, using 
Equation~\eqref{ordmin}, we obtain that $\mxda(r)\leq\mnda(r)$, so
that 
\begin{subequations}
\label{minmaxeq}
\begin{equation}
\label{minmaxeqa}
\mxda(r)=\mnda(r)=-\mxda(r^{-1})=-\mnda(r^{-1}). 
\end{equation}
Similarily we prove that
\begin{equation}
\label{minmaxeqb}
\mxdb(r)=\mndb(r)=-\mxdb(r^{-1})=-\mndb(r^{-1}). 
\end{equation}
\end{subequations}
On the other hand, we already know that in the sum \eqref{58}
 $\exists\; (\mu,\nu):\;\gamma^r_{\mu;\nu}\neq 0$. By the linear independence, also $\gamma^r_{\mu;\nu}+\alpha^r_{\mu;\nu}(A)\neq 0$
and $\gamma^r_{\mu;\nu}+\beta^r_{\mu;\nu}(B)\neq 0$. Therefore,
by \eqref{minmaxeqa} and \eqref{minmaxeqb},
$\mxda(r)=(\mu,\nu)=\mxdb(r)$. Using again the linear independence,
we conclude that all terms in \eqref{58} with the index different
from $(\mu,\nu)$ must vanish.

Summarising, so far we have proven
that an invertible element $r\in\ffun(S^3_{pq\theta})$ and its inverse must have the form
\begin{equation}
r=(\gamma^r_{\mu;\nu}+\alpha^r_{\mu;\nu}(A)+\beta^r_{\mu;\nu}(B))a^\mu b^\nu,\quad
r^{-1}=(\gamma^{r^{-1}}_{-\mu;-\nu}+\alpha^{r^{-1}}_{-\mu;-\nu}(A)+\beta^{r^{-1}}_{-\mu;-\nu}(B))a^{-\mu}b^{-\nu}
\end{equation}
for some  $\mu,\nu\in\Z$, with $\gamma^r_{\mu;\nu}\gamma^{r^{-1}}_{-\mu;-\nu}\neq 0$. Then inserting $r$ and $r^{-1}$ into 
formula~\eqref{Cmnmult} yields
\[
1=e^{-i2\pi\theta\mu\nu}(\gamma^r_{\mu;\nu}+\alpha^r_{\mu;\nu}(A))(\gamma^{r^{-1}}_{-\mu;-\nu}+\alpha^{r^{-1}}_{-\mu;-\nu}(A))(1+\qQ^p_{\mu;-\mu}(A))+
B\,(\mbox{polynomial}(B)).
\]
By linear independence, the term in $B$ vanishes. 
Now by polynomial degree counting, we can conclude that the polynomial in $A$ is of degree zero, and hence so are its factors. 
This yields 
\begin{equation}
\alpha^r_{\mu;\nu}(A)=\alpha^{r^{-1}}_{-\mu;-\nu}(A)=\qQ^p_{\mu;-\mu}(A)=0.
\end{equation}
Repeating the argument for $B$ gives
\begin{equation}
\beta^r_{\mu;\nu}(B)=\beta^{r^{-1}}_{-\mu;-\nu}(B)=\qQ^q_{\nu;-\nu}(B)=0.
\end{equation}
Recalling that $\qQ^p_{\mu;-\mu}(A)=\qQ^q_{\nu;-\nu}(B)=0$ only if $\nu=\mu=0$, we infer that $r=\gamma^r_{0;0}\in\C\setminus\{0\}$.
\end{proof}

Since the Hopf algebra $\ffun(\Z/N\Z)$ contains a non-trivial
 group-like element, in the image of a cleaving map there would
have to be a non-trivial invertible (see preliminaries). Hence
Theorem~\ref{noninv} implies that a cleaving map does not exist:
\begin{corollary}
$\ffun(S^3_{pq\theta})$ is a non-cleft $\ffun(\Z/N\Z)$-comodule algebra.
\end{corollary}

\section{Comodule algebras over the $C^*$-algebras of Heegaard  lens spaces}
\setcounter{equation}{0}

\subsection{K-groups}

For the $K$-theory calculations to come, we utilise a description of the 
Heegaard quantum sphere as a pullback of $U(1)$-$C^*$-algebras,
see the first example of Section~5.2 in~\cite{HKMZ}.
We write $\T$ for the  Toeplitz algebra, and since we will have two 
copies of this algebra, we denote their generating isometries by 
$z_\pm$. 
The corresponding unitaries implementing the crossed products 
$\T\rtimes_{\pm\theta}\Z$ are denoted by $u_\pm$. 
Finally, $Z_+$ and $U_+$
stand for the
two generating unitaries of the noncommutative torus 
$C(S^1)\rtimes_{\theta}\Z$.
With this notation, the pullback structure of 
\begin{equation}
C(S^3_{pq\theta})=\{(a_+,a_-)\in \T\rtimes_\theta\Z\oplus\T\rtimes_{-\theta}\Z:\pi_1(a_+)=\pi_2(a_-)\}
\label{eq:pullback}
\end{equation} 
is given by the following diagram and maps:
\begin{equation}
\vcenter{
\xymatrix@=0pt@R=0.5cm{
&&&C(S^3_{pq\theta})\ar[ldd]\ar[rdd]&&&\\
&&&&&&\\
& &\T\Rt{\theta}\Z \ar[rdd]^(.4){\pi_1}&& \T\Rt{-\theta}\Z\ar[ldd]_(.4){\pi_2}&& \\
z_+\ar@{|->}[dd]&&u_+\ar@{|->}[dd]&&z_-\ar@{|->}[dd]&&u_-\ar@{|->}[dd]\\
&&&\ \ \  C(S^1)\Rt{\theta} \Z\ \ \  &&&\\
Z_+&&U_+&&\!\!\!\!\!\!Z_+^{-1}\!\!\!\!\!\!&&\!\!\!\!\!\!Z_+U_+\,.\!\!\!\!\!\!
}
}
\end{equation}
This is a pullback diagram of $U(1)$-$C^*$-algebras, 
with the natural $U(1)$-action on the $\Z$-parts.
We restrict this action of $U(1)$ to $\Zn$ and consider the pullback diagram obtained by the restriction
of the above one to its $\Zn$-invariant part:
\begin{equation}
\vcenter{
\xymatrix@=0pt@R=0.5cm{
&&&C(L^N_{pq\theta})\ar[ldd]\ar[rdd]&&&\\
&&&&&&\\
& &\T \Rt{\theta} N\Z \ar[rdd]&& \T\Rt{-\theta}N\Z\ar[ldd]&& \\
z_+\ar@{|->}[dd]&&u_+^N\ar@{|->}[dd]&&z_-\ar@{|->}[dd]&&u_-^N\ar@{|->}[dd]\\
&&&\ \ \  C(S^1)\Rt{\theta} N\Z\ \ \  &&&\\
Z_+&&U_+^N&&\!\!\!\!\!\!Z_+^{-1}\!\!\!\!\!\!&&\!\!\!\!\!\!(Z_+U_+)^N.\!\!\!\!\!\!
}
}
\end{equation}
We can use the commutation relations in the noncommutative torus to 
simplify the rightmost map as
$(Z_+U_+)^N=e^{iN(N-1)\pi\theta}Z_+^NU_+^N$.
Introducing the  generators
\begin{equation}
\tilde{z}_+:=z_+,\quad \tilde{u}_+:=u_+^N,\quad \tilde{z}_-:=z_-,\quad
\tilde{u}_-:=u_-^N,\quad
Z:=Z_+,\quad U:=U_+^N,
\end{equation}
we can rewrite this pullback diagram as
\begin{equation}
\vcenter{
\xymatrix@=0pt@R=0.5cm{
&&&C(L^N_{pq\theta})\ar[ldd]_{\mathrm{pr}_1}
\ar[rdd]^{\mathrm{pr}_2}&&&&\\
&&&&&&&\\
& & \T\!\Rt{N\theta}\Z\ar[rdd]^(.4){\tilde{\pi}_1}&& \T\!\!\Rt{-N\theta}\!\Z\ar[ldd]_(.4){\tilde{\pi}_2}&& &\\
\tilde{z}_+\ar@{|->}[dd]&&\tilde{u}_+\ar@{|->}[dd]&&\tilde{z}_-\ar@{|->}[dd]&&\tilde{u}_-\ar@{|->}[dd]&\\
&&&\ \ \  C(S^1)\!\Rt{N\theta} \Z\ \ \  &&&&\\
Z&&U&&\!\!\!\!\!\!Z^{-1}\!\!\!\!\!\!&&e^{iN(N-1)\pi\theta}Z^NU.&\!\!\!\!\!\!\!\!\!\!\!\!
}
}
\end{equation}

For the $K$-theory calculations to come, we need to know the effect of the maps in the pullback diagram on $K$-theory generators. 
These are given by
\begin{align}
K_0(\T\!\Rt{N\theta}\Z)\cong \Z\ni m&\stackrel{\tilde{\pi}_{1\ast}}{\longmapsto} (m,0)\in \Z\oplus\Z\cong K_0(C(S^1)\!\Rt{N\theta} \Z),\nonumber\\
K_0(\T\!\!\Rt{-N\theta}\!\Z)\cong\Z\ni m&\stackrel{\tilde{\pi}_{2\ast}}{\longmapsto} (m,0)\in \Z\oplus\Z\cong K_0(C(S^1)\!\Rt{N\theta} \Z),\nonumber\\
K_1(\T\!\Rt{N\theta}\Z)\cong \Z\ni n &\stackrel{\tilde{\pi}_{1\ast}}{\longmapsto}(0,n)\in \Z\oplus\Z\cong  K_1(C(S^1)\!\Rt{N\theta} \Z),\nonumber\\
K_1(\T\!\!\Rt{-N\theta}\!\Z)\cong \Z\ni n &\stackrel{\tilde{\pi}_{2\ast}}{\longmapsto} (Nn,n)\in\Z\oplus\Z\cong  K_1(C(S^1)\!\Rt{N\theta} \Z).
\end{align}
Inserting these $K$-theory groups and maps into the Mayer-Vietoris 6-term exact sequence (see preliminaries)
\begin{equation}
\vcenter{
\xymatrix{
K_0(C(L^N_{pq\theta}))\ar[r] & 
\stackrel{\phantom{a} }{K_0(\T\!\Rt{N\theta}\Z)\oplus 
K_0(\T\!\!\Rt{-N\theta}\!\Z)}\ar[r]&
\stackrel{\phantom{a} }{K_0(C(S^1)\!\Rt{N\theta} \Z)}\!\ar[d]\\
K_1(C(S^1)\!\Rt{N\theta} \Z)\ar[u]&
\stackrel{\phantom{a} }{K_1(\T\!\Rt{N\theta}\Z)\oplus 
K_1(\T\!\!\Rt{-N\theta}\!\Z)}\ar[l] &K_1(C(L^N_{pq\theta}))\ar[l]
}
}
\end{equation}
yields the exact sequence 
\begin{equation}
\label{halfway}
\vcenter{
\xymatrix{
K_0(C(L^N_{pq\theta}))\ar[rr] && \Z\oplus\Z\ar[rr]^{(m,n)\mapsto (m-n,0)}&&\Z\oplus\Z\ar[d]\\
\Z\oplus\Z\ar[u]&&\Z\oplus\Z\ar[ll]_{(-Nn,m-n)\mapsfrom (m,n)} &&K_1(C(L^N_{pq\theta}))\ar[ll]_-0.
}
}
\end{equation}
Thus we immediately obtain that $K_1(C(L^N_{pq\theta}))=\Z$. Using this information, we can simplify the sequence
\eqref{halfway} to the exact sequence
\begin{equation}
0\rightarrow N\Z\oplus\Z\hookrightarrow\Z\oplus\Z\rightarrow K_0(C(L^N_{pq\theta}))\rightarrow\Z\oplus\Z
\rightarrow\Z\rightarrow 0.
\end{equation}
Consequently, the sequence
\begin{equation}
0\rightarrow N\Z\hookrightarrow \Z\stackrel{f}{\rightarrow}K_0(C(L^N_{pq\theta}))
\rightarrow\Z\rightarrow 0
\end{equation}
is exact. Since $\Z$ is projective,  $K_0(C(L^N_{pq\theta}))=\Im f\oplus \Z$
and $\Im f=\Zn$.
Summarising, we have derived
\begin{theorem}
$K_0(C(L^N_{pq\theta}))=\Zn\oplus \Z$ and $K_1(C(L^N_{pq\theta}))=\Z$.
\end{theorem}

\subsection{The generators of $K_0$}

With the foregoing computation of $K$-groups at hand, we are ready to
prove the main result of this paper.
\begin{theorem} 
\label{thm:not-free}
Let
$L_N:=\{x\in C(S^3_{pq\theta})\;|\;\alpha_{e^{\frac{2\pi i}{N}}}(x)
=e^{\frac{2\pi i}{N}}x\}\subset C(S^3_{pq\theta})$. Then $L_N$ is 
\emph{not 
stably free} as a left $C(L^N_{pq\theta})$-module.
\end{theorem}
\begin{proof}
The $C^*$-algebra $C(S^3_{pq\theta})$ is isomorphic as a 
$U(1)$-$C^*$-algebra  to $C(S^3_{00\theta})$ \cite[Theorem~2.8]{bhms05}. 
The latter is generated by isometries $s$ and $t$
with the $U(1)$-action given by $\tilde{\alpha}_{e^{i\phi}}(s)=e^{i\phi}s$, $\tilde{\alpha}_{e^{i\phi}}(t)=e^{i\phi}t$.
The induced $\Zn$-action can be therefore written as $\Delta_R(s)=s\otimes \tilde{u}$,
$\Delta_R(t)=t\otimes \tilde{u}$, where $\tilde{u}\in C(\Zn)$, $\tilde{u}(e^{\frac{2\pi ik}{N}})=e^{\frac{2\pi ik}{N}}$.
One can immediately check that the formula
\begin{equation}\label{c*strong}
\ell(\tilde{u}^k)=s^{\ast k}\otimes s^k,\quad k\in\{0,\ldots,N-1\}
\end{equation}
defines a strong connection, so that $C(S^3_{pq\theta})$ is a 
$C(\Zn)$-principal comodule algebra (see preliminaries).

Next, let $\C$ be a left $C(\Zn)$-comodule  via 
$\varrho(1)=\tilde{u}\otimes 1$. Then, as explained in 
Subsection~\ref{principal}, we can write
\begin{equation}
L_N\cong C(S^3_{pq\theta})\stackbin[C(\Zn)]{}{\Box}\C \cong
 C(L^N_{pq\theta})ss^*.
\end{equation}
Thus $ss^*\in C(L^3_{pq\theta}(N))$ is 	an idempotent (in fact, projection) representing the $K_0$ class of $L_N$.
On the other hand, we know from the preceding $K$-theory computation that the $\Zn$-part of $K_0(C(L^N_{pq\theta}))$
is generated by the Bass connecting homomorphism applied to the 
$K_1$-class of the unitary $Z\in C(S^1)\rtimes_{N\theta}\Z$.
In other words, $\mathrm{Bass}[Z]\neq 0$ and  $N\mathrm{Bass}[Z]=0$.
To compute $\mathrm{Bass}[Z]$, we lift $Z$ and $Z^{-1}$ to $\tilde{z}_+,\tilde{z}_+^*\in\T\rtimes_{N\theta}\Z$ respectively.
The Bass construction \eqref{p} yields
\begin{equation}
\left(
\begin{array}{cc}
(1,1)&(0,0)\\
(0,0)&(1-\tilde{z}_+\tilde{z}_+^*,0)
\end{array}
\right)=
\left(
 \begin{array}{cc}
1&0\\
0&1
\end{array}
\right)-
\left(
\begin{array}{cc}
0&0\\
0&(\tilde{z}_+\tilde{z}_+^*,1)
\end{array}
\right).
\end{equation}
Finally, using the  pullback description of  $C(S^3_{00\theta})$ in Equation \eqref{eq:pullback}, 
we note that $s$ is expressed as $(z_+u_+,u_-)$. Hence $ss^*$ can be written as
$(z_+z_+^*,1)$. As this element belongs to the $\Zn$-invariant part,  we can rewrite it in terms of the
$C(L^N_{pq\theta})=C(S^3_{pq\theta})^{\Zn}$-generators, which in this instance just means adding $\tilde{}$, so that
$ss^*:=(\tilde{z}_+\tilde{z}_+^*,1)$. 
Now it is clear that
\begin{equation}
\mathrm{Bass}[Z]=2[1]-[L_N]-[1]=[1]-[L_N].
\end{equation}

If $L_N$ were stably free, then there would exist $k,\,m\in \N$ such 
that $L_N\oplus C(L^N_{pq\theta})^k\cong  C(L^N_{pq\theta})^m$ as
modules. Then the foregoing equation would imply 
\begin{equation}
\mathrm{Bass}[Z]=[1]+k[1]-[L_N\oplus C(L^N_{pq\theta})^k ]=(k+1-m)[1].
\end{equation}
However, since $\mathrm{Bass}[Z]\neq 0$, we conclude that $k+1-m\neq 0$. 
Hence $N(k+1-m)\neq 0$ and $N(k+1-m)[1]=0$.
This contradicts the fact that the projection map 
$C(L^N_{pq\theta})\stackrel{\mathrm{pr}_1}{\longrightarrow}
\T\rtimes_{N\theta}\Z$  
takes the identity to the identity inducing
 the map 
\[
K_0(C(L^N_{pq\theta}))\stackrel{\mathrm{pr}_{1*}}{\longrightarrow}
K_0(\T\!\Rt{N\theta}\Z)=\Z[1],\quad
N(k+1-m)[1]\stackrel{\mathrm{pr}_{1*}}{\longmapsto} N(k+1-m)[1]\neq 0.
\]
Hence $L_N$ is not stably free.
\end{proof}

The above theorem shows that the  module $L_N$ associated
to the $C(\Z/N\Z)$-principal comodule algebra $C(S_{pq\theta}))$
is responsible for the torsion part of~$K_0(C(L^N_{pq\theta}))$
and is not stably free. There is a hierarchy of implications:
associated module not stably free $\Rightarrow$ associated module
not free $\Rightarrow$ principal comodule algebra not cleft
$\Rightarrow$ principal comodule algebra not trivial.
In the algebraic part of this paper, we managed to prove by elementary
methods that the $\ffun(\Z/N\Z)$-principal comodule algebra 
$\ffun(S^3_{pq\theta})$ is not cleft. 
Not going beyond algebraic methods,
we could also prove that the associated $\ffun(L^N_{pq\theta})$-module
\[
\mathcal{L}_N:=\left\{x\in \ffun(S^3_{pq\theta})\;\left|\;\alpha_{e^{\frac{2
\pi i}{N}}}(x)=e^{\frac{2\pi i}{N}}x\right\}
\right.\cong \ffun(S^3_{pq\theta})\stackbin[C(\Zn)]{}{\Box}\C 
\]
is not free.
However, to show that $\mathcal{L}_N$ is not stably free, we need to take
 advantage of Theorem~\ref{thm:not-free}.

\begin{corollary} The $\ffun(L^N_{pq\theta})$-module ${\mathcal L}_N$ 
is \emph{not stably free}. 
\end{corollary}
\begin{proof}
Replacing formula \eqref{c*strong} by \eqref{algstrong} defines
a strong connection on $C(\Z/N\Z)$-principal comodule algebra
$C(S^3_{pq\theta})$. As in the proof of Theorem~\ref{thm:not-free},
we can use this strong connection to compute an idempotent matrix
representing the associated $C(L^N_{pq\theta})$-module $L_N$.
It turns out to be
\[
e_N:=\begin{pmatrix} 1-A & p^{-1}zA\\ z^* & p^{-1}A\end{pmatrix}
\in M_2(\ffun(L^N_{pq\theta}))\subset M_2(C(L^N_{pq\theta})).
\]
It follows from Theorem~\ref{thm:not-free}
that, for any non-negative integers $k$ and $l$, there are no 
matrices $v$ and $w$  over $C(L^N_{pq\theta})$ such that
\[
vw=
\left(
\begin{array}{cc}
e_N & 0\\
0 & I_k
\end{array}
\right)
\qquad\mbox{and}\qquad wv=I_l,
\]
where $I_k$ and $I_l$ are identity matrices of the size $k$ and $l$
respectively.
Hence there are no such matrices over $\ffun(L^N_{pq\theta})$.
On the other hand, since $\mathcal{L}_N$ is associated by the same
group-like $\tilde{u}$ as $L_N$, and the same formulae \eqref{algstrong}
define a strong connection on both $\ffun(S^3_{pq\theta})$
and $C(S^3_{pq\theta})$ comodule algebras over $\ffun(\Z/N\Z)=
C(\Z/N\Z)$, we infer that the  idempotent matrix $e_N$  also 
 represents $\mathcal{L}_N$ as an $\ffun(L^N_{pq\theta})$-module.
Combining these two facts, we conclude that 
the $\ffun(L^N_{pq\theta})$-module $\mathcal{L}_N$ is not 
stably isomorphic to $\ffun(L^N_{pq\theta})^l$ for any positive
integer~$l$.
\end{proof}

\section{$U(1)$-prolongations}
\setcounter{equation}{0}

\subsection{Prolongations in the algebraic setting}

By Subsection~\ref{principal}, the prolongation of 
$\ffun(S^3_{pq\theta})$ by $\ffun(U(1))$ is a principal
comodule algebra. Furthermore, 
using \cite[Proposition~4.1]{bz}, one can prove that it is not
cleft
if there are no invertible elements in $\ffun(S^3_{pq\theta})$ other 
then multiples of
identity. Therefore, Theorem~\ref{noninv}
enjoys the following corollary.
\begin{corollary}
The prolongation $\ffun(S^3_{pq\theta})\Box_{\ffun(\Z/N\Z)}\ffun(U(1))$ is a non-cleft $\ffun(U(1))$-comodule
algebra.
\end{corollary}

Now we will try to describe the algebra
 $\ffun(S^3_{pq\theta})\Box_{\ffun(\Z/N\Z)}\ffun(U(1))$
in more detail. 
First, we observe that
the cotensor product 
$\ffun(S^3_{pq\theta}){\Box}_{\ffun(\Z/N\Z)}\ffun(U(1))$ is equal to 
$(\ffun(S^3_{pq\theta})\otimes\ffun(U(1)))^{\co \ffun(\Z/N\Z)}$
for the coaction
\begin{equation}
\ffun(S^3_{pq\theta})\otimes\ffun(U(1))\ni p\otimes h\mapsto  p_{(0)}
\otimes h_{(2)}\otimes p_{(1)}S(\pi(h_{(1)}))\in 
\ffun(S^3_{pq\theta})\otimes\ffun(U(1))\otimes\ffun(\Z/N\Z).
\label{eq:coaction}
\end{equation}
This coaction defines the following $\Z/N\Z$-action:
\begin{gather}\label{alphatilde}
\tilde{\alpha}:\Z/N\Z\ni e^{\frac{2\pi ik}{N}}\mapsto
\tilde{\alpha}_{e^{\frac{2\pi ik}{N}}}\in 
\mathrm{Aut}(\ffun(S^3_{pq\theta})\otimes\ffun(U(1))),\\
\tilde{\alpha}_{e^{\frac{2\pi ik}{N}}}(x\ot h):=
\alpha_{e^{\frac{2\pi ik}{N}}}(x)\ot h(\cdot\; e^{\frac{-2\pi ik}{N}}).
\nonumber
\end{gather}
With this action in mind, we can now write
\[
\ffun(S^3_{pq\theta})\underset{\ffun(\Z/N\Z)}{\Box}\ffun(U(1))
=
(\ffun(S^3_{pq\theta})\otimes\ffun(U(1)))^{\Z/N\Z}
\]

Next,
a straightforward calculation inspired by \cite[Lemma~5.3]{bz}
and taking advantage of
the Hopf $*$-algebra isomorphism
\begin{equation}
\label{invisom}
\psi:\ffun(U(1))\ni u\longmapsto u^N\in {}^{\co\ffun(\Z/N\Z)}\ffun(U(1))
:=\{h\in\ffun(U(1))\;|\;\pi(h_{(1)})\ot h_{(2)}=1\ot h\}
\end{equation}
allows us to prove the following result.
\begin{proposition}
\label{cotnislem}
The  assigments 
\begin{align*}
\ffun(S^3_{pq\theta})\otimes \ffun(U(1))\ni x\otimes h
&\stackrel{\phi}{\longmapsto} x\sw{0}\otimes x\sw{1}\psi(h)\in (\ffun(S^3_{pq\theta})
\otimes\ffun(U(1)))^{\Z/N\Z},\\
(\ffun(S^3_{pq\theta})\otimes\ffun(U(1)))^{\Z/N\Z}
\ni \sum_ix^i\otimes h^i
&\stackrel{\phi^{-1}}{\longmapsto}  \sum_ix^i\sw{0}\otimes\psi^{-1}(S(x^i\sw{1})h^i)\in 
\ffun(S^3_{pq\theta})\otimes \ffun(U(1)),\nonumber
\label{cotnis}
\end{align*}
define mutually inverse isomorphisms of
$*$-algebras.
\end{proposition}

\subsection{Prolongations in the $C^*$-setting}

As above, we can argue that the prolongation of 
$C(S^3_{pq\theta})$ by $\ffun(U(1))$ is a principal
comodule algebra. However, we need to apply a different reasoning
than above to show that it is not cleft. Recall first that,
by Theorem \ref{thm:not-free}, the finitely generated projective  
left $C(L^N_{pq\theta})$-module $L_N$ is not free. Together with the natural identifications
\[
L_N\cong C(S^3_{pq\theta})\underset{\ffun(\Z/N\Z)}{\Box}\C\cong C(S^3_{pq\theta})\underset{\ffun(\Z/N\Z)}{\Box}\ffun(U(1))\underset{\ffun(U(1))}{\Box}\C,
\]
we see that the rightmost module is also not free. Since every module associated with a cleft comodule algebra is necessarily free, we arrive at the following corollary of Theorem \ref{thm:not-free}.
\begin{corollary}
The $\ffun(U(1))$-comodule algebra $C(S^3_{pq\theta})\underset{\ffun(\Z/N\Z)}{\Box}\ffun(U(1))$ is not cleft.
\end{corollary}

To prove an analogue of Proposition~\ref{cotnislem}, we  
 use the identification $C(X,A)\cong A\bar{\ot}C(X)$, where
$X$ is a compact Hausdorff space, $A$ is a unital $C^*$-algebra,
$C(X,A)$ is the algebra of norm-continuous functions, and
$C(X):=C(X,\C)$.
Furthermore, we easily check that the formulae
\begin{align}
(F_1(f))(e^{i\varphi_1},e^{i\varphi_2})
&:=\alpha_{e^{i\varphi_1}}(f(e^{i\varphi_2})),\\
(F_2(g))(e^{i\varphi})
&:=g(e^{i\varphi},e^{i\varphi N}),\\
(G_1(f))(e^{i\varphi_1},e^{i\varphi_2})
&:=\alpha_{e^{i\varphi_1}}(f(e^{i\varphi_2})),\\
(G_2(g))(e^{i\varphi})
&:=g(e^{\frac{-i\varphi}{N}},e^{\frac{i\varphi}{N}}),
\end{align}
define $C^*$-homomorphisms in the  diagram
\[
\xymatrix{
C(U(1),C(S^3_{pq\theta}))\ar@<0.5ex>[r]^-{F_1} &
C(U(1)\times U(1),C(S^3_{pq\theta}))\ar@<0.5ex>[l]^-{G_2}
\ar@<0.5ex>[r]^-{F_2}&C_{\Z/N\Z}(U(1),C(S^3_{pq\theta}))
\ar@<0.5ex>[l]^-{G_1}\\
C(S^3_{pq\theta})\bar{\otimes}C(U(1))\ar[u]_\cong 
\ar@<0.5ex>[rr]
&& 
\left(C(S^3_{pq\theta})\bar{\otimes}C(U(1))\right)^{\Z/N\Z}.
\ar[u]_\cong\ar@<0.5ex>[ll]
}
\]
Here the right bottom corner is defined via an extension of
the action \eqref{alphatilde} to the $C^*$-algebra
$C(S^3_{pq\theta})\bar{\otimes}C(U(1))$. Verifying that
$F_2\circ F_1$ and $G_2\circ G_1$ are mutually inverse maps
yields the desired isomorphism result.
\begin{proposition}
The $C^*$-algebras $C(S^3_{pq\theta})\bar{\otimes}C(U(1))$ and
$(C(S^3_{pq\theta})\bar{\otimes}C(U(1)))^{\Z/N\Z}$ are isomorphic.
\end{proposition}

\section*{Acknowledgements}
This work is part of the EU-project {\sl Geometry and
symmetry of quantum spaces} PIRSES-GA-2008-230836.
 It was  partially supported by the
Polish Government grant 1261/7.PRUE/2009/7. 
The first author gratefully acknowledges the generous
hospitality of the Mathematical Sciences Institute
of the Australian National University that made his
visit to Canberra possible and productive.
The first author is very greatful to Nigel Higson for clarifying
to him the Bass connecting homomorphism in the $C^*$-setting.
The second author was supported by 
the Australian Research Council and the Australian Academy of Sciences. 
The third author thanks Tomasz Brzezi\'nski for discussions.

\end{document}